\documentclass[a4paper,reqno]{amsart}

\usepackage{a4wide}

\usepackage{amssymb,amsmath}
\usepackage{amsthm,eucal}
\usepackage{amsfonts}
\usepackage{mathrsfs,enumerate}
\usepackage{verbatim}

\usepackage{color}

\newcommand{\setR}{{\mathbb R}}
\newcommand{\setN}{{\mathbb N}}
\newcommand{\hopen}{[0,\infty)}
\newcommand{\diff}{\operatorname{D}}
\newcommand{\dv}{\operatorname{div}}
\newcommand{\sol}{u}
\newcommand{\aux}{v}

\newcommand{\mob}{\mathbf m}
\newcommand{\argmin}{\operatorname{argmin}}

\newcommand{\supp}{\operatorname{supp}}
\newcommand{\potential}{{\mathbf V}}
\newcommand{\potentialeps}{\mathbf V_{\kern-2pt\eps}}
\newcommand{\altpotential}{{\mathbf V}}
\newcommand{\nonlin}{{\mathbf N}}
\newcommand{\auxil}{{\mathcal F}}
\newcommand{\auxilmin}{{F}}
\newcommand{\nml}{{\mathbf n}}

\newcommand{\velo}{{\mathbf v}}
\newcommand{\SG}{\mathbf{S}}

\newcommand{\Dom}{\mathrm{Dom}}

\newcommand{\curv}{\gamma}
\newcommand{\dens}{\rho}
\newcommand{\pot}{\varphi}
\newcommand{\Act}{\mathbf A}
\newcommand{\act}[2]{\mathbf{A}^{\kern-2pt#1}_{#2}}

\newtheorem{theorem}{Theorem}
\newtheorem{proposition}{Proposition}[section]
\newtheorem{corollary}[proposition]{Corollary}
\newtheorem{lemma}[proposition]{Lemma}
\newtheorem{definition}{Definition}
\newtheorem{remark}[proposition]{Remark}

\newcommand{\R}{\mathbb{R}}

\newcommand{\N}{\mathbb{N}}

\newcommand{\HH}{\mathscr{H}}

\newcommand{\tauV}{{\kern-3pt\tau}}

\renewcommand{\restriction}[1]{\lower3pt\hbox{$|_{#1}$}}

\newcommand{\Leb}[1]{{\mathscr L}^{#1}} 

\newcommand{\weakto}{\rightharpoonup}
\newcommand{\eps}{{\varepsilon}}
\renewcommand{\to}{\rightarrow}

\definecolor{lblue}{rgb}{0.5,0.5,1}

\renewcommand{\d}{{\mathrm d}}

\newcommand{\Rd}{{\R^d}}
\newcommand{\loc}{{\mathrm{loc}}}

\newcommand{\restr}[1]{\lower3pt\hbox{$|_{#1}$}}
\newcommand{\mass}{\mathfrak m}
\newcommand{\nchi}{{\raise.3ex\hbox{$\chi$}}}

\newcommand{\Mob}{\mob}

\newcommand{\wass}{{{\mathbf W}_{\kern-1pt\mob}}}
\newcommand{\wassdelta}{{{\mathbf W}_{\kern-1pt\mob_\delta}}}
\newcommand{\wassset}[1]{{{\mathbf W}_{\kern-1pt\mob,#1}}}
\newcommand{\energy}{\altenergy} 
\newcommand{\energyo}{\altenergy_0} 
\newcommand{\altenergy}{{\mathbf E}}
\newcommand{\entropy}{\altentropy} 
\newcommand{\altentropy}{{\mathbf U}}

\newcommand{\conv}{{\rm conv}}
\newcommand{\conc}{{\rm conc}}
\newcommand{\incr}{{\rm incr}}
\newcommand{\decr}{{\rm decr}}
\newcommand{\neigh}[1]{\Omega_{[#1]}}
\newcommand{\deltan}{n}
\newcommand{\massOmega}{s_0}
\newcommand{\MobOmega}{{\Mob(s_0)}}
\newcommand{\mmspace}{\admdens} 
\newcommand{\regdens}{X^r(\Omega)}
\newcommand{\regdensx}[1]{X^r({#1})}
\newcommand{\admdens}{X(\Omega)}
\newcommand{\admdensx}[1]{X(#1)}
\newcommand{\admmeas}{\admdens} 
\usepackage[normalem]{ulem} 

\newcommand{\foralltext}{\text{for all }}
\newcommand{\Lipschitz}{LSC}
\newcommand{\GGG}{\color{black}}
\newcommand{\EEE}{\color{black}} 
\newcommand{\gGG}[1]{#1}
\newcommand{\SSS}{\color{red}} 

\newcommand{\nodaniel}{\color{black}}

\author{Stefano Lisini}
\address{Stefano Lisini \\ Dipartimento di Matematica ``F. Casorati''\\ Universit\`a degli Studi di Pavia \\ via Ferrata 1\\ 27100 Pavia \\ Italy}
\email{stefano.lisini@unipv.it}
\author{Daniel Matthes}
\address{Daniel Matthes \\Zentrum Mathematik \\ Technische Universit\"at M\"unchen \\
D-85747 Garching bei M\"unchen\\ Germany}
\email{matthes@ma.tum.de}
\author{Giuseppe Savar\'e}
\address{Giuseppe Savar\'e \\  Dipartimento di Matematica ``F. Casorati''\\ Universit\`a degli Studi di Pavia \\via Ferrata 1\\ 27100 Pavia \\ Italy \footnote{S.L.\ and G.S.\ have been partially supported through PRIN08-grant from MIUR for the project \emph{Optimal transport theory, geometric and functional inequalities, and applications}.}}
\email{giuseppe.savare@unipv.it}

\title[Cahn-Hilliard and thin film equations as gradient flow]
{Cahn-Hilliard and Thin Film equations with nonlinear mobility as gradient flows in weighted-Wasserstein metrics}

\begin{document}

\begin{abstract}
  In this paper, we establish a novel approach to proving existence of non-negative weak solutions
  for degenerate parabolic equations of fourth order,
  like the Cahn-Hilliard and certain thin film equations.
  The considered evolution equations are in the form of a gradient flow for a perturbed Dirichlet energy
  with respect to a Wasserstein-like transport metric,
  and weak solutions are obtained as curves of maximal slope.
  Our main assumption is that the mobility of the particles is a concave function of their spatial density.
  A qualitative difference of our approach to previous ones is that
  essential properties of the solution
  --- non-negativity, conservation of the total mass and dissipation of the energy ---
  are automatically guaranteed by the construction from minimizing movements in the energy landscape.
\end{abstract}

\maketitle

\section{Introduction and statement of main results}
This paper is concerned with the following class of initial-boundary value problems for non-linear fourth order parabolic equations,
\begin{align}
  \label{eq.formalflow}
  &\partial_t\sol = -\dv\big( \mob(\sol) \diff(\Delta\sol -G'(\sol)) \big) \quad \text{in $(0,\infty)\times\Omega$}, \\
  \label{eq.bc}
  &\nml\cdot\diff\sol = 0, \quad
  \nml\cdot\big(\mob(\sol)\diff(\Delta\sol - G'(\sol))\big) = 0
  \quad \text{on $(0,\infty)\times\partial\Omega$}, \\
  \label{eq.ic}
  &\sol(0,x) = \sol_0(x), \quad \text{in }\Omega .
\end{align}
The problem \eqref{eq.formalflow}--\eqref{eq.ic} is posed on a bounded, smooth \EEE convex domain $\Omega\subset\setR^d$. $\nml$ denotes the normal vector field to the boundary $\partial\Omega$. \EEE
The sought solution $\sol:\hopen\times\Omega\to\setR$ is subject to the constraint $0\leq\sol(t,x)\leq M$,
where either $M>0$ is a given number, or $M=+\infty$.
The \emph{mobility} $\mob$ is a non-negative concave function $\mob:(0,M)\to\setR_+$
that vanishes at $0$, and also at $M$ if $M<\infty$.
The mobility $\mob$ and the \emph{free energy} $G:(0,M)\to\R$ are subject to certain regularity assumptions, specified below.
Introducing the \emph{pressure} $P$ satisfying
\begin{align}
  \label{eq.defpressure}
  P'(s)=\mob(s) G''(s),
\end{align}
equation \eqref{eq.formalflow} can be rewritten in the more familiar form
\begin{align}
  \label{eq.formalflow2}
  \partial_t\sol = -\dv\big( \mob(\sol) \diff\Delta\sol \big) + \Delta P(\sol).
\end{align}
Equations of the form \eqref{eq.formalflow} or \eqref{eq.formalflow2} arise,
for instance,
as hydrodynamic approximation to models for many-particle systems in gas dynamics,
and also in lubrication theory.
In particular, the classical Cahn-Hilliard equation
for phase separation in a binary alloy as well as the (de)stabilized thin film equation are of the shape \eqref{eq.formalflow};
we comment on these special cases further below.
The value of the solution $\sol(t,x)$ represents a particle density,
or the fraction of one component of a binary alloy (in the case of the Cahn-Hilliard equation),
or the height of the film (in the case of the thin film equation) at time $t\geq0$ and location $x\in\Omega$.

There is a rich literature on the mathematical structure of Cahn-Hilliard, thin film and related equations.
In particular, the techniques developed in the seminal papers by Elliott and Garcke \cite{ElliGark} and by Bernis and Friedman \cite{BernFrie}
have been proven extremely powerful to carry out existence analysis,
and have been extended by many other authors afterwards.
As a core feature, these techniques allow to replace \eqref{eq.formalflow} by a family of regularized problems with smooth solutions $\sol_\delta$
that satisfy certain bounds which produce the desired constraint $0\le\sol\le M$ in the limit $\delta\downarrow0$.
We emphasize that this bound does not come for free since solutions to fourth order equations do not obey comparison principles in general.

More specifically, in the existence proof for the Cahn-Hilliard equation \cite{ElliGark}, 
the degenerate mobility $\mob$ is replaced by a strictly positive approximation $\mob_\delta$ defined on all $\setR$.
The resulting parabolic problems are non-degenerate and possess global and smooth solutions $\sol_\delta$, which, however, may attain arbitrary real values.
Using additional a priori estimates, it is then shown
that the integral of $\sol_\delta$ in the region where $\sol<0$ or $\sol>M$ converges to zero as $\mob_\delta$ approaches $\mob$,
which yields $0\le\sol(t)\le M$ in the limit.
In fact, a corollary of this method of proof is that solutions to \eqref{eq.formalflow} with a sufficiently degenerate mobility function $\mob$
preserve the \emph{strict} inequalities $0<\sol(t)<M$ for all times $t\ge0$.
This property has been used in the existence proofs for thin film equations \cite{BernFrie,DPasGarkGrun},
where the original mobility is approximated by very degenerate $\mob_\delta$.

The techniques from \cite{BernFrie,ElliGark} rely on the dissipation of certain Lyapunov functionals by solutions to \eqref{eq.formalflow}.
One distinguished Lyapunov functional is a perturbed Dirichlet energy,
which is defined on functions $\sol\in H^1(\Omega)$ with $0\leq\sol\leq M$ by
\begin{align}
  \label{eq.e1}
  \altenergy[\sol] =
  \displaystyle{\frac12\int_\Omega|\diff\sol(x)|^2\d x +  \int_\Omega G(\sol(x))\d x .}
\end{align}
This energy and its dissipation provide regularity estimates.
Another Lyapunov functional introduced in \cite{BernFrie,ElliGark} and since then widely used the literature is
\begin{align}
  \label{eq.altentropy}
  \altentropy[\sol]:= \int_\Omega U(\sol(x))\d x
  \quad\text{with}\quad
  U''(s)=\frac1{\mob(s)}.
\end{align}
If the mobility $\mob(s)$ degenerates sufficiently strongly for $s\downarrow0$ and $s\uparrow M$,
then $\altentropy$ allows to control the solution $\sol$ in the zones where $\sol$ is close to $0$ or $M$.
The functional $\altentropy$ has thus become a key tool for proving the bounds $0\le\sol\le M$.

Here we develop an alternative approach to existence which avoids the
cumbersome discussion of the propagation of the bound $0\le\sol\le M$
\GGG and shows a new interesting variational structure behind
equations of the form \eqref{eq.formalflow}.
\EEE
Our starting point is the classical observation that \eqref{eq.formalflow} is in the shape of a gradient flow for $\altenergy$
on the space of non-negative density functions 
of fixed mass. 
On a purely formal level, the corresponding metric tensor is readily determined:
to a tangential vector $\velo:=\partial_s\dens(0)$
to a smooth curve $\dens:(-\eps,\eps)\to L^1(\Omega)$ of strictly positive densities $\dens(s)$ at $\dens_0=\dens(0)$,
it assigns the length
\begin{align}
  \label{eq.formalmetric}
  \|\velo\|^2 = \int_\Omega |\diff\pot(x)|^2\mob(\dens_0(x))\d x, \quad
  \text{with $\pot$ satisfying} \quad
  \gGG-\dv\big( \mob(\dens_0(x))\diff\pot(x))\big) =
  \velo
  \gGG{\quad\text{in $\Omega$}}
\end{align}
and variational boundary conditions on $\partial\Omega$.

In the particular case of a constant mobility $\mob\equiv1$,
this is simply the dual of the Sobolev seminorm in $W^{1,2}(\Omega)$,
and one can work in the the well know setting of gradient flows in Hilbert spaces, see e.g.\ \cite{BrezisOMM}.
For the linear mobility $\mob(s)=s$, the tensor \eqref{eq.formalmetric} is induced by a non-Hilbertian metric, 
namely the celebrated $L^2$-Wasserstein distance, see \cite{Ott01}. 
In this framework,
weak solutions to specific cases of \eqref{eq.formalflow} have been obtained as curves of steepest descent in the energy landscape of $\altenergy$;
see \cite{GiacOtto,MattMCanSava} for respective results on the Hele-Shaw flow.

For more general mobilities, 
the existence proof presented here seems to be the first based on the gradient flow structure of \eqref{eq.formalflow}
with respect to a metric that is not the $L^2$-Wasserstein distance
\GGG nor a flat Hilbertian one. \EEE
It has been proven only recently by Dolbeault, Nazaret and the third author \cite{DolbNazaSava}
that even for certain \emph{non-linear} mobilities $\mob$,
the formal metric structure indicated in \eqref{eq.formalmetric}
still leads to a genuine metric $\wass$ on the space of positive measures.
One needs to assume, however, that $\mob$ is a concave function
\GGG to get nice analytic and geometric properties of $\wass$: they
have been studied in \cite{CarrLisiSavaSlep} and \cite{LM} and
we review selected results in Section \ref{sct.prelims} below.

\EEE
The goal of this paper is to prove rigorously that weak solutions to \eqref{eq.formalflow}--\eqref{eq.ic} can be obtained
by the variational \emph{minimizing movement/\gGG{JKO} scheme} under suitable conditions on the nonlinear concave function $\mob$.
The terminology \emph{minimizing movement scheme} is due to De Giorgi \cite{DeGiorgi1993}, whereas \emph{JKO scheme} enters in common use after the paper \cite{JordKindOtto}.
\EEE
Preservation of the total mass, dissipation of the energy and, most notably, non-negativity of the density along the solution
are direct consequences of the applied construction:
the solution is a weak limit of time-discrete energy minimizing curves that lie in the convex cone of non-negative densities.
The difficulty of this approach consists in proving a posteriori that the curve of maximal slope indeed corresponds to a weak solution.
For this, a priori estimates resulting from the dissipation of $\altentropy$ are employed.

\subsection{Hypotheses}
We recall that
\begin{center}
  either $M>0$ is a given number, or $M=+\infty$.
\end{center}
All appearing measures are assumed to be absolutely continuous with respect to the Lebesgue measure $\Leb d$, 
and we identify them with their Lebesgue densities on $\Omega$.
The densities have fixed total mass $\mass>0$, and are bounded from above by $M$ if the latter is finite. 
Thus our ambient space will be the metric space $(\admdens,\wass)$ where
\begin{equation}
  \label{eq:9}
  \admdens:=\Big\{u\in L^1(\Omega):0\le u\le M\quad\text{a.e.\ in $\Omega$},\quad \int_\Omega u\,\d x=\mass\Big\}.
\end{equation}
The possible choices for mobility functions $\mob:(0,M)\to\setR_+:=(0,\infty)$ are subject to the following conditions:
\begin{equation}
\label{Mob1}
  \tag{M}
  \begin{aligned}
    &\Mob \text{ is concave,} \quad \Mob \in C^\infty(0,M), \quad
    \Mob>0 \mbox{ in }(0,M),\\
    &\Mob(0):=\lim_{s\downarrow
      0}\Mob(s)=0,\qquad
    \Mob(M):=\lim_{s\uparrow M}\Mob(s) = 0 \qquad
    \GGG \text{if $M<+\infty$}.
  \end{aligned}
\end{equation}
Moreover, we say that the mobility $\mob$ satisfies a {\Lipschitz} condition 
(i.e.\ $\mob$ is Lipschitz and $\mob^2$ is Semi-Convex)
if
\begin{equation}\label{MobL}
  \tag{M-\Lipschitz}
  \sup_{s\in(0,M)}|\mob'(s)|<+\infty \quad\text{and}\quad
  \sup_{s\in(0,M)}\big(-\mob''(s)\mob(s)\big)<+\infty.
\end{equation}
The restriction to concave mobilities in \eqref{Mob1} is necessary,
since only for those, the corresponding metric $\wass$ is well-defined.
Notice that this hypothesis is somewhat opposite to the one made in \cite{ElliGark},
where an asymptotic behaviour $\mob(s)\sim s^\alpha$ for $s\to 0$ with $\alpha\geq1$ has been assumed.
Typical examples for mobility functions with finite $M>0$ are $\Mob(r)=r(M-r)$, or,
more generally, $\Mob(r)=r^{\alpha_0}(M-r)^{\alpha_1}$ with exponents $\alpha_0,\alpha_1\in (0,1]$.
These mobilities satisfy \eqref{MobL} iff $\alpha_0=\alpha_1=1$.
In the case that $M=+\infty$, the mobility $\Mob$ is nondecreasing because it is concave and strictly positive in $(0,+\infty)$.
Typical examples are $\Mob(r)=r^\alpha$ with $\alpha\in (0,1]$;
such a mobility is {\Lipschitz} only in the Wasserstein case $\alpha=1$.


Concerning the free energy $G\in C^2(0,M)$
and the associated pressure $P$ with \eqref{eq.defpressure},
we assume that there exist a constant $C\ge 0$ and an exponent $q>2$ with $q<2d(d-4)$ if $d>4$
such that
\GGG
\begin{equation}\label{Pre1}
  \tag{G}
  \begin{aligned}
    \Mob\, G''\ge -C
    &\text{ in }(0,M),
    \quad P\in C^0([0,M]) &&\quad \text{if }M<\infty,\\
    \Mob \, G''\ge -C(1+\Mob) &\text{ in }(0,\infty),\quad P\in C^0([0,\infty)),\quad
    \lim_{s\to\infty}\frac{P(s)}{s^q+|G(s)|}=0&&\quad\text{if $M=+\infty$}.
  \end{aligned}
\end{equation}
The condition \eqref{Pre1} yields in particular (see \S \ref{subsec:UGproperties})
\begin{equation}
  \label{Pre1cons}
  G\in C^0([0,M])\quad\text{if $M<\infty$};\quad
  G(s)\ge -C(1+s^2)\quad\text{for every $s>0$ if $M=+\infty$}.
\end{equation}
\EEE
Examples for sensible choices of $G$ (and $P$) fulfilling these assumptions are given after the statements of our main results.

\renewcommand{\mu}{u}
\renewcommand{\nu}v

\subsection{The minimizing movement approximation and the existence result for {\Lipschitz} mobilities}
\label{subsec:MM}
The \emph{minimizing movement/{JKO} scheme} is a variational algorithm
\EEE to obtain
a time-discrete approximation (of given step size $\tau>0$) to a curve of steepest descent, see \cite{AmbrGiglSava}.
In the situation at hand, we start from the initial condition $\mu_0\in \admdens$ with $\altenergy[\mu_0]<+\infty$
and define inductively
\begin{align}\label{eq.mm}
  \mu_\tau^0:=\mu_0, \quad
  \mu_\tau^{n+1}:=\argmin \Psi^n_\tau\in \admdens
  \quad\text{where}\quad
  \Psi^n_\tau(\nu):=\frac1{2\tau}\wass(\mu_\tau^n,\nu)^2 + \altenergy[\nu] ,
\end{align}
and we set $\energy[\mu]:=+\infty$ if $\mu\not\in H^1(\Omega)$.
The approximation $\bar \mu_\tau:\hopen\to\mmspace$ is defined by constant interpolation,
using $\bar\mu_\tau(t)=\mu_\tau^n$ for $(n-1)\tau<t\le n\tau$.
\begin{theorem}
  \label{thm.main}
  Assume that $\Omega$ is a smooth \EEE bounded convex open subset of $\R^d$,
  the mobility function  $\mob$ satisfies \eqref{Mob1} and \eqref{MobL},
  and the free energy $G$ satisfies \eqref{Pre1}.

  Then, for any initial condition $\sol_0\in \admdens$ of finite energy $\altenergy[\sol_0]<+\infty$,
  the scheme \eqref{eq.mm} admits time-discrete solutions $\bar\mu_\tau$ for all $\tau>0$.
  For every sequence $\tau_n\downarrow 0$ there exists a subsequence,
  still denoted by $\tau_{n}$,
  and a function $\sol$ satisfying:
  \begin{align}
    \label{spacesol}
    &\sol\in L^2_\loc([0,\infty);H^2(\Omega))\cap
    C^0_w\EEE([0,\infty);H^1(\Omega))\cap {\rm AC}^2_\loc([0,\infty);\admdens)\\
    &\bar\sol_{\tau_{n}}(t) \to \sol(t) \text{ strongly in }
    L^2(\Omega) \text{ and weakly in } H^1(\Omega)
    \quad \text{for all } t\in [0,+\infty), \\
    &\bar\sol_{\tau_{n}} \to \sol \text{ strongly in }
    L^2(0,T;H^1(\Omega))\text{ and weakly in } L^2(0,T;H^2(\Omega))
    \qquad \text{for all $T>0$}.
  \end{align}
  The energy satisfies the bound
  \begin{equation}\label{eq.energydecay}
    \altenergy[\sol(t)]\leq\altenergy[\sol_0] \qquad \foralltext\, t\geq 0,
  \end{equation}
  there exists a decreasing function $\varphi:[0,+\infty)\to \R$ such that
  \begin{equation}\label{eq.energydecay2}
    \varphi(t)\geq\altenergy[\sol(t)] \qquad \foralltext\, t\geq 0,
  \end{equation}
  and
  \begin{equation}\label{eq.energyconv}
    \altenergy[\bar\sol_{\tau_n}(t)]\to\altenergy[\sol(t)] \qquad\mbox{for a.e. } t\geq 0.
  \end{equation}
  Finally $\sol$ satisfies the equation \eqref{eq.formalflow}
  and the boundary conditions \eqref{eq.bc} in the following weak sense:
  \begin{align}\label{eq.weak}
    \int_0^{+\infty} \int_\Omega \partial_t\zeta\,\sol\d x\,\d t
    =  \int_0^{+\infty} \int_\Omega \Delta\sol\dv\big(\mob(\sol)\diff\zeta\big)\d x\,\d t
    - \int_0^{+\infty} \int_\Omega P(\sol) \Delta\zeta\d x\,\d t
  \end{align}
  for every test function $\zeta\in C^\infty_c\big((0,+\infty)\times\overline\Omega\big)$ such that $\diff \zeta\cdot\nml=0$ on $\partial\Omega$.
\end{theorem}
\begin{remark}
  \upshape
  We add a few comments on the previous result: 
  \begin{itemize}
  \item $C^0_w([0,\infty);H^1(\Omega))$ denotes the space of
    weakly continuous curves $u:[0,\infty)\to H^1(\Omega)$.
  \item A curve $u:[0,\infty)\to \admdens$ belongs to ${\rm AC}^2_\loc(0,\infty;\admdens)$ 
    if there exists a function $g\in L^2_\loc([0,\infty))$ such that
    \begin{equation}
      \label{eq:10}
      \wass(u(s),u(t))\le \int_s^t g(r)\, \d r\quad\foralltext\
      0\le s\le t<\infty.
    \end{equation}
  \item The condition $u \in {\rm AC}^2_\loc(0,\infty;\admdens)$ implies that
    \begin{equation}\label{eq.maxest}
      0\leq \sol\leq M, \qquad \int_\Omega \sol(t,x)\d x=\mass \quad \mbox{ for all }  t\geq 0.
    \end{equation}
  \item The condition \eqref{spacesol} implies that
    the initial datum in \eqref{eq.ic} is attained in the sense that
    $\sol(t)$ converges to $\sol_0$ in $\admdens$ with respect to the distance $\wass$
    and weakly in $H^1(\Omega)$ as $t\downarrow0$.
  \item Since $\sol\in L^2_\loc([0,\infty);H^2(\Omega))$ and $\mob$ is {\Lipschitz},
    $\dv\big(\mob(\sol)\diff\zeta\big)\in L^2_{\loc}([0,+\infty);L^2(\Omega))$.
  \item Even in the case $M=+\infty$,
    \eqref{eq.energydecay}, the asymptotic behaviour \eqref{Pre1},
    the lower bound \eqref{Pre1cons} and  
    the Sobolev embedding of $H^2(\Omega)$ in $L^q(\Omega)$
    imply that $P(\sol)\in L^2_\loc([0,+\infty);L^1(\Omega))$.
  \end{itemize}
\end{remark}
The main examples that fits into the framework of Theorem \ref{thm.main} are the classical Cahn-Hilliard equations \cite{CahnHill}:
For the mobility, one chooses $\Mob(r)=r(1-r)$, so $M=1$.
Typical choices for the free energy $G$ are the double well potential,
\begin{align*}
  G(r) = \theta r^2(1-r)^2, 
\end{align*}
yielding a model for the phase separation for a binary alloy,
\begin{align}
  \label{eq.cahn2}
  \partial_t \sol = -\dv\big( \sol(1-\sol)\diff\Delta\sol \big) + \theta\Delta \sol^2(1-\sol)^2;
\end{align}
or the function
\begin{align*}
  G(r) = \theta \big( r\log r + (1-r)\log(1-r) \big), \qquad \theta \in \R,
\end{align*}
that lead to an equation for the volume fraction of one component in binary gas mixture,
\begin{align}
  \label{eq.cahn}
  \partial_t \sol = -\dv\big( \sol(1-\sol)\diff\Delta\sol \big) + \theta\Delta\sol.
\end{align}
See \cite{CahnHill,CahnTayl} for a derivation of \eqref{eq.cahn2} and\eqref{eq.cahn},
and \cite{ElliGark} for a related existence analysis.

\subsection{The existence result for more general mobilities}
The assumption \eqref{MobL} about mobility can be weakened to condition \eqref{Mob4} below,
at the price that the weak solution to \eqref{eq.formalflow}--\eqref{eq.ic}
is no longer obtained as a curve of steepest descent in $\wass$,
but appears as the weak limit of such curves in metrics satisfying \eqref{MobL}.
\begin{theorem}
  \label{thm.main2}
  Assume that $\Omega$ is a smooth \EEE bounded convex open subset of $\R^d$,
  the mobility function  $\mob$ satisfies \eqref{Mob1}, 
  and the free energy $G$ satisfies \eqref{Pre1}.
  In addition, assume that $\mob$ satisfies
  \begin{equation}\label{Mob4}
    \tag{{M$_{1/2}$}}
    \lim_{s\downarrow0}s^{1/2}\Mob'(s)=0, \quad\text{and, if $M<\infty$, also} \quad
    \lim_{s\uparrow M}(M-s)^{1/2}\Mob'(s)=0.
  \end{equation}
  Then, for any initial condition $\sol_0\in \admdens$ of finite energy $\altenergy[\sol_0]<\infty$,
  there exists a function $\sol\in L^2_\loc(0,\infty;H^2(\Omega))\cap C^0_w([0,\infty);H^1(\Omega))$
  satisfying the constant mass and maximum estimate \eqref{eq.maxest}, the energy bound \eqref{eq.energydecay} \EEE
  and the weak formulation \eqref{eq.weak}  of equation \eqref{eq.formalflow} with the boundary conditions \eqref{eq.bc}.
  The initial condition \eqref{eq.ic} is met in the sense
  that $\sol(t)$ weakly converges to $\sol_0$ in $H^1(\Omega)$ as $t\downarrow0$.
\end{theorem}
The first condition in \eqref{Mob4} is needed to give a meaning to the gradient of $\mob(\sol)$ in the weak formulation \eqref{eq.weak},
particularly on the set $Z=\{(t,x)\in(0,+\infty)\times\Omega:\sol(t,x)=0\}$. 
We briefly indicate the problem:
Since $\sol(t)\in H^2(\Omega)$ for a.e.\ $t\geq0$,
and $\sol(t)$ satisfies homogeneous Neumann boundary conditions,
the Lions-Villani-estimate on square roots \cite{LionVill} (see Lemma \ref{lem.Sobolev} in the Appendix) 
states that
\begin{align}
  \label{eq.villani0}
  \big\|\diff\sqrt{\sol(t)}\big\|_{L^4(\Omega)}^2 \leq C \|\diff^2\sol(t)\|_{L^2(\Omega)} .
\end{align}
Thus $\sol\in L^2_\loc(0,\infty;H^2(\Omega))$ in combination with \eqref{Mob4} implies that $\Mob(\sol)\in L^2_{\loc}(0,\infty;H^1(\Omega))$.
In fact, in the proof it turns out that $\diff\Mob(\sol)=0$ a.e.\ on the set $Z$.
A similar reasoning applies to the zero set of $\diff\sqrt{M-\sol}$ if $M<\infty$.
Therefore, it suffices to evaluate the second spatial integral in \eqref{eq.weak}
only on the subset $\{0<\sol(t)<M\}\subset\Omega$,
thus avoiding to discuss the singularity of $\Mob'(s)$ at $s=0$ or $s=M$.
Unfortunately, \eqref{eq.villani0} cannot be extended to obtain $L^{2p}$ estimates on roots $\sol(t)^{1/p}$ with $p>2$, as is easily seen.
Without further a priori estimates, there is apparently no way to remove condition \eqref{Mob4}.

The main example in the framework of Theorem \ref{thm.main2}
is the -- (de)stabilized -- lubrication or thin film equation,
where one chooses $M=+\infty$ and $\Mob(r)=r^\alpha$ with $1/2<\alpha \leq 1$.
The equation is
\begin{align}
  \label{eq.lubrication}
  \partial_t \sol = -\dv\big( \sol^\alpha \diff\Delta\sol \big) + \kappa \Delta \big( \sol^\beta \big),
\end{align}
where $\sol:\Omega\times(0,\infty)\to [0,+\infty)$
describes the height of a thin viscous liquid film on a substrate,
moving under the influence of surface tension;
the lower order perturbation is typically attributed to van der Waals forces
or similar intermolecular interactions.
The destabilized case corresponds to $\kappa<0$
while for $\kappa>0$, the contribution has a stabilizing effect.

The existence theory of \eqref{eq.lubrication} for the unperturbed flow $\kappa=0$
is fairly well understood \cite{DPasGarkGrun}.
In particular, the Hele-Shaw equation obtained for $\alpha=1$ has been analyzed thoroughly
as a gradient flow of the Dirichlet functional in the $L^2$-Wasserstein metric, see e.g. \cite{CarlUlus,GiacOtto,MattMCanSava}.
The perturbed flow has essentially been treated in $d=1$ dimensions only,
see e.g. \cite{LaugPugh,SlepPugh},
but some results (e.g. on the blow-up behavior of solutions) are available
also in multiple dimensions \cite{EvenGalaKing}.
For non-linear mobilities $\mob$, the equation's gradient flow structure has apparently not been exploited
for rigorous analytical treatment before.

In order to obtain \eqref{eq.lubrication} from \eqref{eq.formalflow}, one would like to choose
\begin{align*}
  G(r) =\kappa  \frac{\beta}{(\beta-\alpha)(\beta-\alpha+1)}\ r^{\beta-\alpha+1}
\end{align*}
in the definition of the energy \eqref{eq.e1}.
This is, however, only possible for certain regimes of $\beta$ and $\kappa$:
\begin{itemize}
\item If $1\leq\beta\leq\alpha+1$, then
  \GGG $G$ 
  satisfies
  \eqref{Pre1}
  \EEE
  for all $\kappa\in\setR$.
\item \GGG If $\alpha<1$ and
  \EEE $\alpha+1<\beta$ \GGG
  or $\alpha=1$ and $\beta>2$ with $\beta<2d/(d-4)$ if $d>4$, 
  then 
  \emph{provided} that $\kappa\geq0$, i.e., the perturbation must be stabilizing.
\item If $\beta<1$, then there is no way to accommodate the perturbation into our framework.
\end{itemize}

\subsection{Key ideas of the existence proof}
The discrete approximation scheme in \eqref{eq.mm} provides a family of piecewise constant approximate solutions $\bar\sol_\tau:\hopen\to H^1(\Omega)$.
Weak convergence towards a limit curve $\bar u:\hopen\to H^1(\Omega)$ along a sequence $\tau_n\downarrow0$ is easily obtained,
using the machinery developed in \cite{AmbrGiglSava}.
The difficulty lies in identifying the weak limit $\bar u$ as a weak solution to \eqref{eq.formalflow}--\eqref{eq.ic}.

For mobilities satisfying \eqref{MobL},
a semi-discrete version of the weak formulation \eqref{eq.weak} is derived by variational methods,
i.e., we use suitable perturbations of the minimizers $\sol_\tau^n$ in each step of the scheme \eqref{eq.mm}.
Our variations of the $\sol_\tau^n$ are obtained by applying an auxiliary gradient flow to them.
Specifically, in order to arrive at \eqref{eq.weak},
we would like to use variations in the direction of the flow generated by the functional
\begin{align*}
  \altpotential[\sol] := \int_\Omega V\sol\,\d x
\end{align*}
with a given test function $V\in C^\infty(\overline\Omega)$ satisfying homogeneous Neumann boundary conditions on $\partial\Omega$. \EEE
To motivate this particular choice, assume for the moment that the test function $\zeta$ factors as $\zeta(t,x)=\psi(t)V(x)$.
Then the left-hand side of \eqref{eq.weak} can formally be read as
\begin{align*}
  \int_0^T\int_\Omega \sol(t,x)\psi'(t)V(x)\d x\,\d t
  = -\int_0^T \psi(t)\frac\d{\d t}\altpotential[\sol(t)]\d t,
\end{align*}
i.e., as the temporal derivative of the functional $\altpotential$ along the sought gradient flow for $\altenergy$.
Further, the ``flow interchange'' Lemma \ref{lem.flowinterchange}, taken from \cite{MattMCanSava}, says that this expression can equally be understood
as the temporal derivative of the functional $\altenergy$ along the gradient flow of $\altpotential$.
Thus, variations of the minimizers for \eqref{eq.mm} along the flow of $\altpotential$ are expected to provide a form of \eqref{eq.weak}.

Unfortunately, $\altpotential$ itself is not a suitable choice for carrying out estimates,
since the gradient flow generated by $\altpotential$ is not regular enough to apply the flow interchange lemma.
In particular, the functional $\altpotential$ is not geodesically $\lambda$-convex
in the metric $\wass$ for any $\lambda \in\R$ (see \cite{CarrLisiSavaSlep}). \EEE
As a matter of fact, the trajectories of the gradient flow of $\altpotential$ with respect to $\wass$
are formally given by the solutions of the non-linear conservation law
\begin{align}
  \label{eq.claw1}
  \partial_{\gGG t} \aux_{\gGG t} = \dv\big( \mob(\aux_{\gGG t})\diff V \big).
\end{align}
These solutions are expected to develop shocks in finite time.
To circumvent this technical problem, we consider a modification of $\altpotential$,
$$ \altpotential_\eps [\sol] := \altpotential [\sol] + \eps\altentropy [\sol], $$
where $\altentropy$ is defined in \eqref{eq.altentropy},
that induces the following viscous regularization in \eqref{eq.claw1}
\begin{align*}
  \partial_{\gGG t} \aux_{\gGG t} = \dv\big( \mob(\aux_{\gGG t})\diff V \big) +\eps \Delta \aux_{\gGG t}.
\end{align*}
\EEE

For {\Lipschitz} mobilities, the viscous regularized flow generated by $\altpotential_\eps$
with respect to $\wass$ is $\lambda_\eps$-convex
and our strategy goes through.
For more general mobilities, even the viscous flow lacks convexity.
This makes it necessary to perform further approximations:
we replace the mobility function by {\Lipschitz}-ones,
obtain a weak formulation \eqref{eq.weak} for the corresponding flows,
and then pass to the non-{\Lipschitz} limit.

Even with the discrete version of \eqref{eq.weak} at hand,
we still need to facilitate sufficiently strong compactness to pass to the time-continuous limit $\tau\downarrow0$.
Our key estimate is obtained from the dissipation of the functional $\altentropy$ in \eqref{eq.altentropy}
along solutions of \eqref{eq.formalflow}.
A direct calculation shows that the dissipation term provides a bound in $L^2_\loc(0,\infty;H^2(\Omega))$.
The rigorous proof of this $H^2$-bound is obtained by another application of the strategy above:
we interchange flows and calculate the variations of $\altenergy$ with respect to perturbations of the minimizers
in the direction of the gradient flow generated by $\altentropy$.
This time, the strategy goes through smoothly since the auxiliary functional $\altentropy$,
which generates the heat flow with homogeneous Neumann boundary condition,
is geodesically convex with respect to the considered metric $\wass$. \EEE
\medskip

The paper is organized as follows.
Section \ref{sct.prelims} below provides the essential definitions for the measure-theoretic formulation of the problem.
In Section \ref{sct.apriori} we have collected a variety of technical results
that are applied in subsequent sections to obtain a priori estimates on the semi-discrete approximation $\bar\sol_\tau$.
Sections \ref{sct.Lipschitz} and \ref{sct.general} are devoted to the proofs of Theorems \ref{thm.main} and \ref{thm.main2}, respectively.
There, we follow the strategy outlined above.

\section{Preliminaries}
\label{sct.prelims}

\GGG
\subsection{Basic assumptions}
\label{subsec:basic}
Here and in the rest of this paper, we will always assume that
\begin{equation}
  \label{eq:5}
  \tag{$\Omega$-conv}
  \Omega\quad\text{is a convex, smooth \EEE \GGG and bounded open set of $\Rd$, \quad
    $\mass_\Omega:=\frac\mass{|\Omega|}\in (0,M)$,}
\end{equation}
where $M\in (0,+\infty]$ characterizes the domain of the mobility
function $\mob$ as in \eqref{Mob1}.
We will always assume that $\mob$ is a mobility function satisfying
\eqref{Mob1} and $G$ is a free energy density satisfying \eqref{Pre1}.
\subsection{\EEE
Notation: admissible and regular densities}
\GGG
As in \eqref{eq:9},
for a given mobility $\mob$ and a mass $\mass$ satisfying \eqref{eq:5}
we introduce the sets of \emph{admissible} and \emph{regular
  densities} on $\Omega$
\begin{align}
  \admdens &:= \Big\{ \dens\in L^1(\Omega) \,\Big|\, 0\le\dens\le M,
  \,\int_\Omega\dens \,\d x=\mass \Big\},\\
  \label{eq.regulardensity}
  \regdens &:= \Big\{ \dens\in C^\infty(\Omega)\, \Big|\,
  0<\inf\dens\le \sup\dens<M,\quad\int_\Omega\dens\,\d x=\mass
  \Big\}.
\end{align}
\GGG
Since we will keep fixed the mobility $\mob$ (and its domain of definition $(0,M)$) and
the total mass $\mass$, we will omit to indicate the explicit
dependence of the above spaces from these two parameters.

We often identify an element $u\in \admdens$ with the nonnegative
measure $\mu=u\Leb d$ in $\Rd$ supported in $\overline \Omega$ and we will
consider weak convergence of sequences in $\admdens$ in the sense of distributions of $\mathscr
D'(\Rd)$:
\begin{equation}
  \label{eq:11}
  u_n\weakto u\quad\text{in }\mathscr D'(\Rd)\quad\Leftrightarrow\quad
  \lim_{n\to\infty}\int_\Omega u_n\varphi\,\d x=\int_\Omega
  u\varphi\,\d x\quad\text{for all }\varphi\in C^\infty_{\rm c}(\Rd).
\end{equation}
\begin{remark}
  \label{rem:convergence}
  \upshape
  Since $u_n$ are nonnegative with fixed total mass, we could also
  equivalently consider test functions $\varphi\in C^0_{\rm c}(\Rd)$
  in \eqref{eq:11}; when $M<\infty$, $\admdens$ is a convex and bounded
  subset of $L^\infty(\Omega)$ and \eqref{eq:11} also coincides with
  the weak-$\star$ convergence in the latter space.
\end{remark}
For every extended-valued real functional $\auxil:\admdens\to (-\infty,+\infty]$ we denote
by $\Dom(\auxil)$ its proper domain $\Dom(\auxil):=\{u\in
\admdens:\auxil(u)<\infty\}$. $\auxil$ is called proper if
$\Dom(\auxil)$ is not empty.
\EEE
%
\GGG

We will consider curves in $\regdens$: they are maps
$\gamma:[0,1]\to\admdens$ which we will also identify with functions
$\gamma:[0,1]\times\Omega\to [0,M]$ such that $\gamma(t,\cdot)\in
\regdens$ for every $t\in [0,1]$.
We say that
\begin{equation}
  \label{eq:16}
  \gamma:[0,1]\to\regdens\text{ is regular if}\qquad
  t\mapsto \gamma(t,x) \in C^1([0,1])\quad\foralltext x\in \Omega.
\end{equation}
In a similar way, a functional
\begin{equation}
\auxil:\regdens\to\R\text{ is regular if}\qquad
t\mapsto \auxil[\gamma(t)]\in C^1([0,1])\quad\text{for every regular
  curve }\gamma,
\label{eq:3}
\end{equation}
and a map $\SG:[0,\infty)\times \regdens\to\regdens$ is regular if the curves
\begin{equation}
\SG(\cdot,u),\
\SG(t,\gamma(\cdot)) \text{ are regular for every
$u\in \regdens$, $t\ge0$, and for every regular curve $\gamma$.}
\label{eq:18}
\end{equation}

\EEE

\subsection{Survey: \GGG weighted transport distances\EEE}
We shall now review the \GGG weighted transport distances \EEE
$\wass$ introduced in \cite{DolbNazaSava} (see also \cite{CarrLisiSavaSlep} and \cite{LM}) \EEE
without going into details about their formal definition.
\GGG When $M=\infty$ they could in fact be pseudo-metrics, i.e.\ they satisfy all the
axioms of the usual notion of distance except for the fact that
the value $+\infty$ may be attained; nevertheless, even in the case
$M=\infty$ the next proposition shows that
the restriction of $\wass$ to the sublevels of
the convex functional (recall \eqref{eq.altentropy})
\begin{equation}
  \label{eq:2}
  \altentropy [u]:=
    \int_\Omega U(u)\,\d x
  \qquad
  \text{where}\quad
  U''(s)=\frac1{\mob(s)},\quad
  U(\mass_\Omega)=U'(\mass_\Omega)=0,
\end{equation}
is a finite distance.
Notice moreover that, besides $\mob$, $\wass$ also depends on the
domain $\Omega$: we will denote it by $\wassset\Omega$ when we want to
stress this dependence. In particular, \EEE
for every $\delta>0$ we will also sometimes consider the $\delta$-neighborhood $\neigh\delta$ of $\Omega$
\begin{equation}\label{eq:14}
  \neigh \delta := \Omega+\delta\mathbb{B}^d = \big\{ x\in\setR^d \big| \,\operatorname{dist}(x,\Omega)<\delta \big\}.
\end{equation}
\begin{proposition}\label{prp.wm}
  \GGG In the setting of \S \ref{subsec:basic}, \EEE
  the pseudo-metric $\wass$
  on the space $\mmspace$ has the following properties:
  \begin{enumerate}[(a)]
  \item
    \GGG For every $u_n,u\in \admdens$,
    \begin{equation}
      \label{eq:12}
      \lim_{n\to\infty}\wass(u_n,u)=0\quad\Leftrightarrow\quad
      u_n\weakto u\quad\text{in }\mathscr D'(\Rd)
    \end{equation}
    according to \eqref{eq:11} (but see also remark
    \ref{rem:convergence}). \EEE
  \item For every $c\ge0$ the \GGG sublevels of $\altentropy$
    \begin{equation}
      \label{eq:4}
      \Big\{ \mu\in\mmspace\,\big|\, \altentropy[\mu]\le c 
      \Big\}\quad\text{are compact metric spaces w.r.t.\ }\wass.
    \end{equation}
    \EEE
  \item
    %
    For every decreasing sequence of convex sets $\Omega^n$ converging to $\Omega$,
    \EEE
    if two sequences $\mu_0^n,\,\mu_1^n\in
    \GGG \admdensx{\Omega^{n}}
    \EEE $
    converge to $\mu_0$ and $\mu_1$ in the sense of distributions
    respectively, i.e.
    \begin{equation}
      \label{eq:13}
      \lim_{n\to\infty}\int_{\Omega^{n}} u_i^n\varphi\,\d x=\int_\Omega
      u_i\varphi\,\d x\quad\foralltext\varphi\in C^\infty_{\rm
        c}(\Rd),\quad i=0,1,
    \end{equation}
    then
    \begin{equation}\label{lsc}
      \wassset\Omega(\mu_0,\mu_1) \le \liminf_{n\to +\infty}\wassset{\Omega^{n}}(\mu_0^n,\mu_1^n).
    \end{equation}
  \item If $\curv:[0,1]\to\regdens$ is a regular curve according to
    \eqref{eq:16}
    and $\pot:[0,1]\to H^1(\Omega)$ is the corresponding curve of weak solutions to
    \begin{align}
      \label{eq.Neumann}
      -\dv\big(\mob(\curv)\diff\pot\big) = \partial_s\curv \quad\text{in $\Omega$},
      \qquad
      \nml\cdot\diff\pot = 0 \quad\text{on $\partial\Omega$},
    \end{align}
    then the $\wass$-distance between $\mu_0=\curv(0)$ and
    $\mu_1=\curv(1)$ is bounded as follows:
    \begin{align}
      \label{eq.wmabove}
      \wass(\mu_0,\mu_1)^2 \le \int_0^1 \int_\Omega \mob(\curv) |\diff\pot|^2\d x\,\d s.
    \end{align}
  \item \label{bull.approxgeodesic}
    \GGG
    Let $\mu_i
    \in \admdens$, $i=0,1$, be given with 
    $\wass(\mu_0,\mu_1)<\infty$. \EEE
    For every decreasing sequence of smooth convex sets $\Omega_n$ converging to $\Omega$ as $n\to \infty$,
    such that $\Omega_n \supset \neigh{\delta_n}$ for a vanishing sequence $\delta_n$,
    there exists
    a sequence of regular curves (``approximate geodesics'')
    $\curv_n:[0,1]\to \regdensx{\Omega_n}$,
    such that:
    \begin{itemize}
    \item
      $\curv_n(0)$ and $\curv_n(1)$ converge to
      $\mu_0$ and $\mu_1$, respectively, in $L^1(\Rd)$ as
      $n\to\infty$ and for every proper and lower
      semicontinuous convex integrand
      $F:[0,M]\to [0,\infty]$
     \begin{equation}
        \label{eq:15}
        \lim_{n\to\infty}\int_{\Omega_n} F(\gamma_n(i,x))\,\d
        x=\int_\Omega F(\mu_i)\,\d x\qquad i=0,1.
      \end{equation}
   \item if $\pot_n:[0,1]\to H^1(\Omega_n)$ are the corresponding curves of weak solutions to \eqref{eq.Neumann} on $\Omega_n$,
      then
      \begin{align}
        \label{eq.wmbelow}
        \wass(\mu_0,\mu_1)^2 = \lim_{n\to\infty} \int_0^1
        \int_{\Omega_n} \mob(\curv_n)|\diff\pot_n|^2\d x\, \d s .
      \end{align}
    \end{itemize}
    \EEE
  \end{enumerate}
\end{proposition}
No properties of the metric $\wass$ other than those listed above will be used in the sequel.
Notice that \eqref{eq.wmabove}\&\eqref{eq.wmbelow} establish the connection between the metric $\wass$
and the formal definition of the metric tensor given in \eqref{eq.formalmetric}.

\subsection{The entropy and energy functionals}
\label{subsec:UGproperties}
In this section, we derive some relevant properties of the entropy and the energy
densities $U,G$ introduced in \eqref{eq.altentropy} and \eqref{eq.e1}.
For definiteness, we make the following specific choice for the function $U$ in \eqref{eq.altentropy}:
\begin{align}
  \label{eq.entropy}
  U(s):=\int_{s_0}^s \frac{s-r}{\mob(r)}\d r ,\quad
  \massOmega:=\mass_\Omega=\frac\mass{|\Omega|}.
\end{align}
\EEE
\begin{lemma}
  \label{lem.entropy}
  The entropy functional $\entropy$ is lower semi-continuous
  with respect to the weak convergence \eqref{eq:11},
  and satisfies the following bounds
  \begin{align}
    \label{eq.entropyest1}
    0 \leq \entropy[\mu] \leq
    C(1 + \|\sol\|_{L^2(\Omega)}^2), \qquad \foralltext \,\mu\in \Dom(\entropy).
  \end{align}
  \EEE
  The constant $C$ above only depends on $\Omega$, $\mass$, and
  $\MobOmega=\Mob(\mass_\Omega)$. 
\end{lemma}
\begin{proof}
  Lower semi-continuity is a consequence of the convexity of $U$,
  which, in turn, follows from $U''(s)=1/\mob(s)>0$ for every $s\in (0,M)$.
  The lower bound in \eqref{eq.entropyest1} follows from non-negativity of $U$,
  indeed convexity of $U$ and \eqref{eq:2} yield that $\mass_\Omega$ is a minimum for $U$ and $U(\mass_\Omega)=0$.
  \EEE
  For showing the upper bound in \eqref{eq.entropyest1}, first note that
  \begin{align}
    \label{eq.mbelow}
    \mob(s) \geq \begin{cases}
      \frac{\MobOmega}{\massOmega}s & \mbox{if }s\le s_0\\
      \frac{\MobOmega}{M-\massOmega}(M-s) & \mbox{if }s > s_0, \quad M<+\infty \\
      \MobOmega& \mbox{if }s > s_0, \quad M=+\infty ,
    \end{cases}
  \end{align}
  by concavity of $\mob$.
  Thus for $C_0:=\frac {\massOmega^2}\MobOmega$ and $C_M:=\frac{(M-\massOmega)^2}\MobOmega$
  \begin{align*}
    U(s) \leq \left\{\begin{array}{ll}
        \frac{\massOmega}{\MobOmega}(s(\log
        s-\log(\massOmega)-1)+\massOmega)\leq
        C_0
        & \mbox{if } 0\leq s\leq \massOmega, \\
        \frac{M-\massOmega}{\MobOmega}((M-s)(\log(M-s)-\log(M-\massOmega)-1)+M-\massOmega)
        \leq
        C_M
        & \mbox{if } s>\massOmega, M<\infty, \\
        \frac1{2\MobOmega}(s-\massOmega)^2 \leq \frac1{2\MobOmega}s^2 & \mbox{if } s>s_0, M=\infty.
      \end{array}\right.
  \end{align*}
  Now \eqref{eq.entropyest1} follows by the boundedness of $\Omega$.
\end{proof}
\GGG
Concerning the function $G$,
we
decompose its second derivative $G''$ into the difference of its positive
and negative part
\begin{equation}
  \label{eq:25}
  L:=G''=L^+-L^-,\qquad
  L^-:=-\min(G'',0),\quad L^+:=\max(G'',0).
\end{equation}
Fixing
$s_0\in (0,M)$ (e.g.\ $s_0=\mass_\Omega$ as before) and assuming without loss of
generality
that $G(s_0)=G'(s_0)=0$ (recall that the integral of elements in
$\admdens$ is fixed to be $\mass$) we have the decomposition
\begin{equation}
  \label{eq:21}
  G=G_{\conv}+G_\conc,\qquad
  G_\conv(s)=\int_{s_0}^s L^+(r)(s-r)\,\d r,\quad
  G_\conc(s)=-\int_{s_0}^s L^-(r)(s-r)\,\d r,
\end{equation}
and the corresponding one
\begin{equation}
  \label{eq:23}
  P=P_{\incr}+P_{\decr},\quad
  P_\incr(s)=\int_{s_0}^s {L^+(r)}{\Mob(r)}\,\d r,\quad
  P_\decr(s)=-\int_{s_0}^s {L^-(r)}{\Mob (r)}\,\d r.
\end{equation}
\eqref{eq:21} and the upper bound $L^-\le C(1+1/\Mob)=C(1+U'')$ of \eqref{Pre1}
yield
\begin{equation}
  \label{eq:24}
  G_\conv(s)\ge G_\conv(s_0)=0=G_\conc(s_0)\ge G_\conc(s)\ge
  -C(1+s^2)\quad\foralltext s\in (0,M),
\end{equation}
proving the lower bound in \eqref{Pre1cons}.
It follows immediately from the lower bounds in \eqref{Pre1} that
$P_\decr$ is Lipschitz continuous and $G_\conc$ is continuous in
$[0,M)$ (and also in $M$ if $M<\infty$) since it is concave and
bounded from below.

In order to check the continuity of $G_\conv$ in $0$ (the same
argument applies to $M$ when $M<\infty$), let us first observe that
$P_\incr'=L^+\Mob$ is integrable
around $0$ since $P_\incr=P-P_\decr$ is locally bounded around $0$ by
\eqref{Pre1}. Recalling \eqref{eq.mbelow} we easily get for $0<s<s_0$
\begin{align*}
  G_\conv(s)=&\int_s^{s_0}\frac{P_\incr'(r)}{\mob (r)}(r-s)\,\d r\le
  \frac{s_0}{\Mob(s_0)}\int_s^{s_0}{P_\incr'(r)}\,\d r\le -P_\incr(0).
\end{align*}
Since $G_\conv$ is convex we conclude that it has a right limit at
$0$.

With \eqref{Pre1cons} and the above remarks at our disposal, we can
obtain simple lower bounds on the
the energy functional $\energy$ defined in \eqref{eq.e1}.
\EEE
\begin{lemma}[Basic properties of $\energy$]
  \label{lem.energyprops}
  The functional $\energy$ is
  \GGG bounded from below in the space $\admdens$ and lower semi-continuous with respect to
  the distributional convergence \eqref{eq:11} in the space
  $\admdens$. \EEE
  Finally the following estimate
  holds
  \begin{align}
    \label{eq.boundbelow}
    \frac18\|\sol\|_{H^1(\Omega)}^2 + \GGG\int_\Omega G_\conv(\sol)\,\d x
    \EEE \leq
    \energy[\mu] + \gGG\energyo  \qquad \foralltext \, \mu \in
    \admmeas,\EEE
  \end{align}
  \GGG where $G_\conv$ has been defined by \eqref{eq:21}-\eqref{eq:25}
  \EEE
  and
  the constant $\energyo$ only depend on $\Omega$, the mass $\mass$, the dimension $d$ and the
  function $G$.
\end{lemma}
\begin{proof}
  To begin with, we prove the estimate \eqref{eq.boundbelow}.
  We recall the following Gagliardo-Nirenberg (\cite{Gagliardo59}, \cite{Nirenberg59}) interpolation inequality
  \begin{align}\label{eq.gagliardo}
    \|\sol\|_{L^2(\Omega)} \leq C_1 \|\diff\sol\|_{L^2(\Omega)}^\theta\|\sol\|_{L^1(\Omega)}^{1-\theta}
    + C_2\|\sol\|_{L^1(\Omega)}, \qquad \sol\in H^1(\Omega)
  \end{align}
  where $\theta=d/(d+2)$ and the constants $C_1$, $C_2$ only depend on $\Omega$ and $d$.
  In our specific case of $\sol\in\admdens\cap H^1(\Omega)$ we have, for every $\eps>0$,
  \begin{align}\label{eq.gagliardo2}
    \|\sol\|_{L^2(\Omega)} \leq C_1 \|\diff\sol\|_{L^2(\Omega)}^\theta \mass^{1-\theta}
    + C_2\mass \leq \eps\|\diff\sol\|_{L^2(\Omega)} +C_3(\eps)\mass
  \end{align}
  where $C_3(\eps):= ({2C_1}/{\eps})^{1/(1-\theta)}+C_2$.
  In particular
  \begin{align}\label{eq.gagliardo3}
    \|\sol\|_{L^2(\Omega)}^2 \leq 2\eps^2\|\diff\sol\|_{L^2(\Omega)}^2 +2C_3(\eps)^2\mass^2.
  \end{align}
  Using the decomposition \eqref{eq:21}, the lower bound
  \eqref{eq:24}, and \eqref{eq.gagliardo3},
  for the constant $C$ in \eqref{eq:24} we have
  \begin{equation}
    \label{eq.d2byenergy}
    \begin{split}
      \energy[\mu]
      & \geq \frac12 \|\diff\sol\|_{L^2}^2 - C\big( |\Omega| + \|\sol\|_{L^2}^2 \big) + \int_\Omega G_\conv(\sol)\d x \\
      & \geq \Big(\frac12-2\eps^2 C\Big) \|\diff\sol\|_{L^2}^2 - C( |\Omega|\nobreakspace + 2C_3(\eps)^2\mass^2) + \int_\Omega G_\conv(\sol)\d x .
    \end{split}
  \end{equation}
  Choosing $\eps^2=1/(8C)$ in \eqref{eq.d2byenergy} and using again \eqref{eq.gagliardo3} with $\eps^2=1/2$
  we obtain \eqref{eq.boundbelow} with the constant
  $\energyo :=  C( |\Omega|\nobreakspace + 2C_3(1/(2\sqrt{2C}))^2\mass^2) + 1/4 C_3(1/\sqrt{2})^2\mass^2$.
  Boundedness of $\energy$ from below is an immediate consequence of \eqref{eq.boundbelow}, recalling that $G_\conv$ is non-negative.
\EEE

  In order to prove lower semi-continuity,
  assume that a sequence 
  $\mu_k\in\mmspace$ converges 
  to a limit $\mu\in\mmspace$ according to \eqref{eq:11}.
  It is not restrictive to assume that $\mu_k\in H^1(\Omega)$
  and that
  \GGG $\sup_{k\to\infty}\energy[\mu_k]<+\infty.$ \EEE
  By estimate \eqref{eq.boundbelow}
  the sequence $\sol_k$ is bounded in $H^1(\Omega)$.
  Hence, up to subsequences, $\sol_k$ converges weakly in $H^1(\Omega)$, converges strongly in $L^2(\Omega)$,
  and converges pointwise $\Leb{d}$-a.e. to $\sol$.

  By \eqref{eq:24} and Fatou's Lemma we have that
  \begin{equation}\label{lsc12}
    \liminf_{k\to+\infty}\int_\Omega G_\conc(\sol_k) +C(1+\sol_k^2)\d x \geq
    \int_\Omega G_\conc(\sol)+C(1+\sol^2)\d x.
  \end{equation}
  The $L^2(\Omega)$ strong convergence of $\sol_k$ and concavity of $G_\conc$ yield
   \begin{equation}\label{lsc13}
    \limsup_{k\to+\infty}\int_\Omega G_\conc(\sol_k) +C(1+\sol_k^2)\d x \leq
    \int_\Omega G_\conc(\sol)+C(1+\sol^2)\d x.
  \end{equation}
  From \eqref{lsc12} and \eqref{lsc13} it follows
  \begin{equation}\label{lsc1}
    \lim_{k\to+\infty}\int_\Omega G_\conc(\sol_k)\d x =
    \int_\Omega G_\conc(\sol)\d x .
  \end{equation}
  \EEE
  Second, by Fatou's Lemma,
  it follows that
  \begin{equation}\label{lsc2}
    \liminf_{k\to+\infty}\int_\Omega \GGG G_\conv\EEE (\sol_k)\d x
    \geq \int_\Omega \GGG G_\conv(\sol)\EEE \d x .
  \end{equation}
  Finally, since $\diff\sol_k$ converges weakly in $L^2(\Omega)$ to $\diff\sol$,
  \begin{equation}\label{lsc3}
    \liminf_{k\to+\infty}\int_\Omega |\diff\sol_k|^2\d x \geq \int_\Omega |\diff\sol|^2\d x .
  \end{equation}
  The lower semi-continuity of $\energy$ follows from \eqref{lsc1}, \eqref{lsc2} and \eqref{lsc3}.
\end{proof}
\GGG
We will denote by $\energy_{\rm min}$ the minimum value (depending on
$\Omega,\mass,G$) of $\energy$
on $\admmeas$.
\EEE
Notice that estimate \eqref{eq.entropyest1} in combination with \eqref{eq.boundbelow} yields
\begin{align}
  \label{eq.entropyest}
  0
  \leq \entropy[\mu] \leq C(\energy[\mu]+\gGG\energyo),
  \qquad \foralltext \,\mu\in \Dom(\energy), 
\end{align}
with some constant $C$ only depending on $\Mob(s_0)$, $\Omega$, $\mass$ and $G$.

\section{A priori estimates}
\label{sct.apriori}

\subsection{Semi-discrete approximation}
\label{sct.discrete}
We begin by invoking a result from \cite{AmbrGiglSava} that guarantees the well-posedness of the minimizing movement scheme \eqref{eq.mm},
i.e. the existence of the semi-discrete curves $\bar \mu_\tau$ and their compactness for vanishing step size $\tau\downarrow0$.
\begin{proposition}
  \label{prp.discrete}
  \GGG
  In the setting of \S \ref{subsec:basic},
  \EEE
  for every $\mu_0\in\GGG\admmeas\cap \EEE\Dom(\energy)$
  and $\tau>0$ there exists a sequence $\{\mu_\tau^n\}_{n\geq0}$
  satisfying \eqref{eq.mm} and
  the following energy estimate:
  \begin{align}
    \label{eq.estimate1}
    \energy[\mu^N_\tau] + \frac1{2\tau} \sum_{n=1}^N \wass(\mu^n_\tau,\mu^{n-1}_\tau)^2 \leq \energy[\mu_0]
    \qquad \foralltext N\in\N.
  \end{align}
  Moreover, for every sequence $\tau_n\downarrow 0$ there exists a subsequence,
  still denoted by $\tau_{n}$,
  and a continuous limit curve $\mu:[0,+\infty)\to \gGG\admmeas$ 
  such that $\bar\mu_{\tau_{n}}(t)$ converges weakly to $\mu_t$
  \GGG in $H^1(\Omega)$ \EEE
  for every $t\geq0$.
  The curve $\mu$ is globally $1/2$-H\"older continuous 
  \begin{align}
    \label{eq.estimate2}
    \wass(\mu_t,\mu_s)\leq\big(2(\energy[\mu_0]\GGG -\energy_{\rm min})\EEE\big)^{1/2}|s-t|^{1/2} \quad \foralltext s,t\in[0,+\infty).
  \end{align}
  The curve $t\mapsto u_t$ satisfies  \eqref{eq.energydecay},
  \eqref{eq.energydecay2}, \eqref{eq.energyconv} and \eqref{eq.maxest}.
  \EEE
\end{proposition}
Proposition \ref{prp.discrete} is obtained by combining the results
from Chapters 2 and 3 (see in particular Sections 2.1, 2.2 and Corollary 3.3.4) of \cite{AmbrGiglSava}.
The properties of $\energy$ proven in Lemma \ref{lem.energyprops} are
sufficient to apply this general theory
\GGG and the uniform upper bound on $\energy [\bar u_\tau]$ given by
\eqref{eq.estimate1}
improves the pointwise convergence of $\bar u_{\tau_n}$ with respect
to $\wass$ to the weak convergence in $H^1(\Omega)$. \EEE
It should be remarked that we do not claim uniqueness of solutions,
even on this discrete level,
\GGG except in the case when $\altenergy$ is a convex functional. \EEE

\subsection{Flow interchange lemma}
For the derivation of $\tau$-independent a priori estimates on the interpolations $\bar\mu_\tau$,
we employ the device of the \emph{flow interchange lemma}, which has been proven in \cite{MattMCanSava}.
Before reviewing the lemma and its proof, we
\GGG recall the definition of $\lambda$-flow in the metric space
$\admmeas$ given in \cite{DaneSava}.
\EEE
%
\begin{definition}
  Let $\auxil :\admmeas\to (-\infty,+\infty]$ be a proper lower semi-continuous functional and $\lambda\in\R$.
  A continuous semi-group $\SG^t:\Dom(\auxil)\to \Dom(\auxil)$,
  $t\ge0$,
  is a \emph{$\lambda$-flow for $\auxil$}
  if it satisfies the \emph{Evolution Variational Inequality} (EVI)
  \begin{equation}\label{EVIgen}
    \frac12\limsup_{h\downarrow0}\bigg[ \frac{\wass(\SG^h( \mu),\nu)^2-\wass(\mu,\nu)^2}h\bigg]
    + \frac{\lambda}{2}\wass(\mu,\nu)^2 \GGG+\auxil[\mu]\EEE
    \leq \auxil[\nu] ,
  \end{equation}
  for all measures $\mu,\nu\in \Dom(\auxil)$
  \GGG with $\wass(\mu,\nu)<+\infty$. \EEE
\end{definition}
\GGG
Recall that a continuous semigroup $\SG$ on a set $D\subset \admdens$ is a family of
maps $\SG^t:D\to D$, $t\ge0$, satisfying
\begin{equation}
  \label{eq:6}
  \SG^{t+s}(\mu)=\SG^t(\SG^s(\mu)),\quad
  \lim_{t\downarrow0}\wass(\SG^t(\mu),\mu)=0\quad
  \foralltext \mu\in D.
\end{equation}
Notice that the continuity of $\SG$ is already coded in
\eqref{EVIgen}: it is sufficient to choose
$v:=u$ in
\eqref{EVIgen}.

\eqref{eq:6} and the triangle inequality yields
\begin{equation}
  \label{eq:7}
  \wass(\SG^t(\mu),\mu)<+\infty\quad\foralltext \mu\in D,\ t\ge0;
\end{equation}
in particular the ``$\limsup$'' in \eqref{EVIgen} is well defined.
\EEE
%
\begin{lemma}[Flow interchange Lemma \cite{MattMCanSava}]\label{lem.flowinterchange}
  Assume that $\SG_\auxil$ is a $\lambda$-flow for the proper, lower
  semi-continuous functional $\auxil$ in $\admmeas$ and
  let $\mu^n_\tau$ be a $n$-th step approximation constructed by the minimizing movement scheme \eqref{eq.mm}.
  \GGG If $\mu^n_\tau\in \Dom(\auxil)$
  \EEE
  then
  \begin{align}
    \label{eq.flowinterchange}
    \auxil[\mu^n_\tau] - \auxil[\mu^{n-1}_\tau]
    \leq \tau \liminf_{h\downarrow0} \bigg(\frac{\energy[\SG_\auxil^h(\mu^n_\tau)]- \energy[\mu^n_\tau]}h\bigg)
    - \frac\lambda2 \wass\big(\mu^n_\tau,\mu^{n-1}_\tau\big)^2 .
  \end{align}
\end{lemma}
\begin{proof}
  (C.f. \cite{MattMCanSava})
  By definition of $\mu_\tau^n$ as a minimizer in \eqref{eq.mm},
  \begin{align*}
    \frac1{2\tau}\wass\big(\mu_\tau^n,\mu_\tau^{n-1}\big)^2 + \energy[\mu^\tau_n]
    \le \frac1{2\tau}\wass\big(\SG_\auxil^h(\mu_\tau^n),\mu_\tau^{n-1}\big)^2 + \energy[\SG_\auxil^h(\mu^\tau_n)]
  \end{align*}
  holds for every $h>0$.
  This implies
  \begin{align*}
    - \tau \liminf_{h\downarrow0} \big[ h^{-1} \big( \energy[\SG_\auxil^h(\mu^n_\tau)] - \energy[\mu^n_\tau] \big) \big]
    \le \frac12 \limsup_{h\downarrow0}  \big[h^{-1}\wass\big(\SG_\auxil^t(\mu_\tau^n),\mu_\tau^{n-1}\big)^2-\wass(\mu_\tau^n,\mu_\tau^{n-1})^2\big].
  \end{align*}
  To conclude \eqref{eq.flowinterchange} from here,
  apply \eqref{EVIgen} with the choices $\mu=\mu_\tau^n$ and $\nu=\mu_\tau^{n-1}$.
\end{proof}

\subsection{Eulerian calculus}
In order to apply the flow interchange Lemma \ref{lem.flowinterchange} with a particular auxiliary functional $\auxil$,
we need to \GGG exhibit the \EEE
associated semigroup $\SG$
\GGG
(usually given implicitly as the solution to a nonlinear evolution
equation)
\EEE
and to verify that it is indeed a
$\lambda$-flow,
i.e., it satisfies the EVI \eqref{EVIgen} with a finite constant $\lambda$.
A very general strategy to attack this problem is the Eulerian calculus for transportation metrics,
that has been developed by the third author in \cite{DaneSava}, based on earlier work by Otto and Westdickenberg \cite{OttoWest}.
Similar to the flow interchange estimate,
the basic idea is to simplify estimates by exchanging two time-like
derivatives.
We also need that $\SG$ can be suitably approximated by semigroups on
smooth densities: here is the relevant definition.
\EEE
%
\begin{definition}
  \label{def:mollification}
  \GGG
  Let us fix a nonnegative vanishing sequence $\delta_n$,
  let $\Omega_n \supset \neigh{\delta_n}$, a decreasing sequence of smooth convex sets converging to $\Omega$,
  \EEE
  \GGG
  let $\auxil:\admdens\to(-\infty,+\infty]$ be
  proper and l.s.c.\ functionals, and let $\SG$ be a semi-group on \GGG
  $\Dom(\auxil)\subset \admdens$.

  We say that
  $\{\auxil_n,\SG_n\}_{n\in\N}$
  is a \emph{family of mollifications} for $\auxil,\SG$
  if
  \begin{enumerate}[(a)]
  \item
    $\auxil_n:\regdensx{\Omega_n}\to \R$ is a regular functional
    according to \eqref{eq:3}.
    \item
      Each $\SG_n$ is a regular semi-group on
      $\regdensx{\Omega_n}$ according to \eqref{eq:18}. 
    \item
      For every $\rho_0,\rho_1\in \Dom(\auxil)$ with
      $\wass(\rho_0,\rho_1)<\infty$
      the regular densities $\gamma_n(i)$ given as in (e) of
      Proposition \ref{prp.wm} satisfy
      \begin{equation}
        \label{eq:20}
        \lim_{n\to\infty}\auxil_n[\gamma_n(i)]=\auxil[\rho_i]\qquad i=0,1.
      \end{equation}
    \item
        For every sequence
  $\dens_n\in \Dom(\auxil_n)$ converging to
  $\dens\in \Dom(\auxil)$ in $L^1(\Rd)$ as $n\to\infty$ with
  $\auxil_n[\dens_n]\to \auxil[\dens]$ we have
  $\SG^t_n\dens_n\weakto\SG^t\dens$ in $\mathscr D'(\Rd)$ and
  $\liminf_{n\to\infty}\auxil_n[\SG_n^t\dens_n]\ge
  \auxil[\SG^t\dens]$ for all
  $t\ge0$.
  \end{enumerate}
\end{definition}
Let $\auxil,\auxil_n$ and $\SG,\SG_n$ as in the definition above.
For a given $n\in \N$
consider a regular curve $\curv_n:[0,1]\to\regdensx{\Omega_n}$ and for every $h\ge0$,
introduce $\curv_n^h:[0,1]\to\regdensx{\Omega_n}$ --- the \emph{$h$-perturbation of $\curv_n$} ---
by
\begin{align*}
  \curv_n^h(s) = \SG_n^{hs}\curv_n(s).
\end{align*}
\GGG Notice that the $(h,s)\mapsto \curv_n^h(s,x)$ is of class $C^1$
in $[0,\infty)\times [0,1]$ thanks to the regularity of $\SG_n$.
\EEE
Also, introduce the \emph{action} of the perturbed curves
\begin{equation}
  \label{eq:19}
  \act hn(s) = \int_{\Omega_n} |\diff\pot^h_n(s)|^2\mob(\curv^h_n(s))\,\d x,
\end{equation}
where the $\pot^h_n(s)\in H^1(\Omega_n)$ form a $s$-differentiable family of solutions to the associated Neumann problems
\begin{align}
  \label{eq.Neumann2}
  -\dv\big(\mob(\curv^h_n(s,\cdot))\diff\pot^h_n(s,\cdot)\big) = \partial_s\curv^h_n(s,\cdot) \quad \text{in $\Omega_n$},
  \qquad
  \nml\cdot\diff\pot^h_n(s,\cdot) = 0 \quad\text{on $\partial\Omega_n$}.
\end{align}
These Neumann problems are solvable because $\curv^h_n$ is a regular
curve of densities in $\regdensx{\Omega_n}$;
in particular, the mass is constant, and thus $\partial_s\curv^h_n(s,\cdot)$ has vanishing average on $\Omega_n$.
The following result is essentially an adaptation of Theorem 2.2 in \cite{DaneSava} to the situation at hand.
\begin{lemma}
  \label{lem.eulerian}
  Under the hypotheses and with the definitions above,
  assume that $\inf_n\auxil_n\ge \auxilmin>-\infty$ and $h\mapsto\auxil_n[\SG_n^h\dens]$ are non-increasing for
  every $\dens\in \regdensx{\Omega_n}$,
  and
  \GGG there exists $\lambda\le 0$ (independent of $n$ and of
  the considered curves $\curv_n$) such that
  \EEE the inequality
  \begin{align}
    \label{eq.eulerian}
    \frac12\partial_h\act hn(s) + s\lambda\act hn(s) \le -\partial_s\auxil_n[\curv^h_n(s)]
  \end{align}
  holds 
  for all $s\in[0,1]$ and all $h\ge0$.
  Then 
  $\SG$
  is a $\lambda$-flow for $\auxil$.
\end{lemma}
\begin{proof}
  The core idea is to
  \GGG prove and integrated form of \eqref{EVIgen} by estimating
  the perturbed action \eqref{eq:19} starting from a family of approximating geodesics \EEE
  -- provided by (\ref{bull.approxgeodesic}) in Proposition \ref{prp.wm} --
  connecting two given admissible measures
  $\mu,\nu\in \Dom(\auxil)\subset \admmeas$.
  Without loss of generality, we assume that $\auxil_n$ are non-negative and
  \GGG $\lambda<0$; the case $\lambda=0$ follows by obvious
  modifications. \EEE

  Given $\mu,\nu\in \Dom(\auxil)\subset \admmeas$ at finite distance, and
  a family of approximating geodesic $\curv_\deltan$ on $\regdensx{\Omega_{\deltan}}$ between $\nu$ and $\mu$
  in the sense of (\ref{bull.approxgeodesic}) in Proposition \ref{prp.wm}, \EEE
  define $\mu_\deltan^0=\curv_\deltan(0)$ and $\mu_\deltan^1=\curv_\deltan(1)$.
  Multiply \eqref{eq.eulerian} by $e^{2\lambda h s}$ and integrate with respect to $s\in[0,1]$;
  this gives
  \begin{align*}
    \frac12 \frac{\d}{\d h} \int_0^1 e^{2\lambda h s}\act h\deltan(s)\d s
    &\le - \int_0^1 e^{2\lambda h s}\partial_s\auxil_n[\curv^{\gGG h}_\deltan(s)]\d s \\
    &= \auxil_n[\mu_\deltan^0] - e^{2\lambda h}\auxil_n[\SG_{\deltan}^h\mu_\deltan^1]
    + \int_0^1 2\lambda h e^{2\lambda h s}
    \auxil_n[\curv_\deltan^h(s)]\d s \\
    & \le \auxil_n[\mu_\deltan^0] - e^{2\lambda h}\auxil_n[\SG_{\deltan}^h\mu_\deltan^1] ,
  \end{align*}
  since $\lambda<0$ while $\auxil_n$ is non-negative.
  Next, integrate with respect to $h\in[0,H]$,
  which yields
    \begin{align*}
    \frac12 \int_0^1 e^{2\lambda H s}\act H\deltan(s)\d s + \frac{1-e^{2\lambda H}}{-2\lambda} \auxil_n[\SG_{\deltan}^H\mu_\deltan^1]
    \le \frac12\int_0^1 \act {\gGG0}\deltan(s)\d s + H\auxil_n[\mu_\deltan^0],
  \end{align*}
  where we also used the fact that $h\mapsto\auxil_n[\SG_{\deltan}^h\mu_\deltan^1]$ is non-increasing.
 Further, a reparametrization of $s\mapsto\curv_\deltan^h(s)$ similar to that used in \cite[Lemma 5.2]{DaneSava}
  yields
  \begin{align*}
    \frac{e^{-2\lambda H}-1}{-2\lambda H}\wass(\mu_\deltan^0,\SG_{\deltan}^H\mu_\deltan^1)^2
    \le \int_0^1 e^{2\lambda H s}\act H\deltan(s)\d s.
  \end{align*}
  In summary, we have
  \begin{align*}
    \frac{e^{-2\lambda H}-1}{-2\lambda H}\wass(\mu_\deltan^0,\SG_{\deltan}^H\mu_\deltan^1)^2
    + \frac{1-e^{2\lambda H}}{-2\lambda} \auxil_n[\SG_{\deltan}^H\mu_\deltan^1]
    \le \frac12\int_0^1 \act {\gGG0}\deltan(s)\d s + H\auxil_n[\mu_\deltan^0].
  \end{align*}
  By our choice of $\curv_\deltan$, \GGG (e) of Proposition
  \ref{prp.wm} and \eqref{eq:20} yield on one hand that
  \begin{align*}
    \lim_{\deltan\to\infty}\int_0^1 \act 0\deltan(s)\d s =
    \wass(\mu,\nu)^2,\qquad
    \lim_{\deltan\to\infty}\auxil_n[\mu_\deltan^0]=\auxil[\nu].
  \end{align*}
  On the other hand, we know by Proposition \ref{prp.wm} and the
  properties listed in Definition \ref{def:mollification}
  that $u^1_\deltan$ converges to the density $\mu$ in $L^1(\Rd)$ and
  \EEE
  \begin{align*}
    \auxil[\SG^H\mu]\le\liminf_{\deltan\to\infty}\auxil_n[\SG_{\deltan}^H\mu_\deltan^1],
    \qquad
    \wass(\SG^H\mu,\nu)^2 \le \liminf_{\deltan\to\infty}\wass(\SG_{\deltan}^H\mu_\deltan^1,\mu_\deltan^0)^2.
  \end{align*}
  Altogether, this yields the inequality
  \begin{align*}
    \frac{e^{-2\lambda H}-1}{-2\lambda H}\wass(\SG^H\mu,\nu)^2
    + \frac{1-e^{2\lambda H}}{-2\lambda} \auxil[\SG^H\mu]
    \le \frac12\wass(\mu,\nu)^2 + H\auxil[\nu],
  \end{align*}
  from which the EVI property \eqref{EVIgen} is deduced after division by $H>0$ in the limit $H\downarrow0$.
\end{proof}

\section{Proof of Theorem \ref{thm.main}}
\label{sct.Lipschitz}
\GGG Throughout this section we use the notation introduced in \S
\ref{subsec:MM}. \EEE

\subsection{$H^2$-regularity and strong convergence}

The goal of the following is to prove:
\begin{proposition}
  \label{prp.strong}
  In the setting of \S \ref{subsec:basic} 
  each solution $u^n_\tau$ of the minimizing movement scheme
  \eqref{eq.mm} satisfies
  \begin{equation}\label{eq.h2regularity}
    \sol_\tau^n\in H^2(\Omega) \qquad \foralltext \, n \in \N, \ \tau>0,
  \end{equation}
  and the following bound holds for the piecewise constant interpolant $\bar\sol_\tau$,
  \begin{equation}\label{eq.h2}
    \int_0^T \|\bar\sol_\tau(t)\|_{H^2(\Omega)}^2\d t \leq CT(\energyo+\energy[\mu_0]) < +\infty
    \qquad \foralltext\,T>0,
  \end{equation}
  with a constant $C$ independent of $\tau>0$.
  Moreover, every sequence $\tau_k\downarrow0$ contains a subsequence (still denoted by $\tau_k$) such that
  \begin{align}
    \label{eq.strongconv}
    &\bar\sol_{\tau_k} \to \sol \text{ strongly in }
    L^2(0,T;H^1(\Omega))
    \qquad \foralltext\,T>0, \\
    \label{eq.weakconv}
    &\bar\sol_{\tau_k} \to \sol \text{ weakly in }  L^2(0,T;H^2(\Omega)) \qquad \foralltext\,T>0.
  \end{align}
\end{proposition}
The proof of Proposition \ref{prp.strong} rests on the fact that the densities $\sol_\tau$
are $\tau$-uniformly bounded in $L^2(0,T;H^2(\Omega))$ for arbitrary $T>0$.
As motivation for the \GGG arguments \EEE below,
we provide the relevant formal calculations in the case $G\equiv0$:
assuming that $\sol$ is a smooth solution to \eqref{eq.formalflow} satisfying \eqref{eq.bc},
differentiation of the entropy functional $\altentropy[\sol(t)]$ introduced in \eqref{eq.altentropy}
yields
\begin{align*}
  \frac{\d}{\d t}\altentropy[\sol(t)]
  &= \int_\Omega U'(\sol(t))\partial_t\sol(t)\d x
  = -\int_\Omega U'(\sol(t)) \dv\big( \Mob(\sol(t)) \diff\Delta\sol(t)\big) \d x \\
  &= \int_\Omega U''(\sol(t))   \Mob(\sol(t)) \diff\sol(t)\cdot \diff\Delta\sol(t) \d x
  = \int_\Omega \diff\sol(t)\cdot \diff\Delta\sol(t)\d x\\
  &= - \int_\Omega \big(\Delta\sol(t)\big)^2\d x
\end{align*}
because of the identity $U''(s)=1/\mob(s)$.
Consequently, $\altentropy[\sol(t)]$ is decreasing with respect to $t$,
and (still formally),
\begin{align*}
  \altentropy[\sol(T)] + \int_0^T\int_\Omega (\Delta\sol)^2 \d x\,\d t \leq \altentropy[\sol^0] .
\end{align*}
\GGG
Taking also into account the contribution of $G$ and the convexity of
$\Omega$, \EEE
one ends up with an estimate of the form \eqref{eq.h2}.
The goal for the rest of this section is the rigorous proof of this estimate.



We wish to apply the {flow interchange Lemma} \ref{lem.flowinterchange} with $\auxil=\altentropy$.
To this end, we need to identify the associated semi-group $\SG$,
with $\SG^t\aux$ given by the smooth solution $\aux_t$ to the Neumann problem
\begin{align}
  \label{eq.heat}
  \partial_t \aux_t = \Delta\aux_t \quad \mbox{in $\Omega$}, \qquad
  \nml\cdot\diff\aux_t = 0 \quad \mbox{on $\partial\Omega$}, \qquad
  \aux_0=\aux.
\end{align}
\begin{lemma}
  \label{lem.heatconvex}
  The semi-group $\SG$ induced by solutions $\aux_t$ of the problem \eqref{eq.heat} on $\regdens$
  extends to a $0$-flow $\SG$ for $\entropy$. \EEE
\end{lemma}
This fact is a special case of a more general result proven in \cite[Theorem 6.1]{CarrLisiSavaSlep}.
We provide the relevant calculations for the specific situation of Lemma \ref{lem.heatconvex} as we shall refer to it later.
\begin{proof}
  We wish to apply Lemma \ref{lem.eulerian}.
  In order to define a family of mollifiers $\{\auxil_n,\SG_n\}_{n\in \N}$ for $\entropy$, $\SG$,
  we consider a sequence of domains $\Omega_n:=\neigh {\delta_n}$ for some vanishing sequence $\delta_n>0$
    and we define $\auxil_n(\sol):=\entropy_n[\sol]:=\int_{\Omega_n}U(\sol(x))\d x$ and the heat-semigroup $\SG_n$ on
    $\Omega_n$ with homogeneous Neumann boundary conditions. \EEE
  The contraction properties of the heat flow imply that convergence of the initial conditions in $L^1$
  imply the same convergence of the solution at any time $h>0$
  \GGG and it is easy to verify all properties required in Definition \ref{def:mollification}. \EEE

  \GGG Let regular curves $\curv_n:[0,1]\to\regdensx{\Omega_n}$ be given.
  By classical parabolic theory, the $h$-perturbed curves $\curv^h_n:[0,1]\to\regdensx{\Omega_n}$ are well-defined for any $h\ge0$,
  and for every $s\in[0,1]$, the function $(h,x)\mapsto\curv^h_n(s,x)$ is a classical solution to
  \begin{align}
    \label{eq.perturbedgeodesic}
    \partial_h \curv^h_n = s\Delta\curv^h_n \quad \text{in $\Omega_n$}, \qquad
    \diff\curv^h_n\cdot\nml = 0 \quad \text{on $\partial\Omega_n$}, \qquad
    \curv^0_n(s)=\curv_n(s).
  \end{align}
  \EEE
  We need to verify the principal estimate \eqref{eq.eulerian} of Eulerian calculus,
  which reads in the situation at hand \GGG (we will omit to indicate $n$
  in the following) \EEE
  as follows:
  \begin{align}
    \label{eq.eulerentropy}
    \frac12 \int_\Omega \partial_h \big[\mob(\curv^h)|\diff\pot^h|^2\big]\,\d x
    \le - \int_\Omega \partial_s \big[ U(\curv^h)\big] \,\d x.
  \end{align}
  Using the definition of $\pot^h$ in \eqref{eq.Neumann2} and its boundary conditions,
  the right-hand side evaluates after integration by parts to
  \begin{align*}
    - \int_\Omega U'(\curv^h)\partial_s\curv^h\,\d x
    = - \int_\Omega U''(\curv^h)\diff\curv^h\cdot\big(\mob(\curv^h)\diff\pot^h\big)\,\d x
    = - \int_\Omega \diff\curv^h\cdot\diff\pot^h\,\d x
    = \int_\Omega \curv^h\Delta\pot^h\,\d x
  \end{align*}
  since $U''(s)\mob(s)=1$ by definition of $U$.
  For the $h$-derivative of the action, we find
  \begin{equation}
    \label{eq.actionderivative}
    \begin{aligned}
      \frac12 \int_\Omega& \partial_h \big[\mob(\curv^h)|\diff\pot^h|^2\big]\,\d x
      = \frac12 \int_\Omega \partial_h\mob(\curv^h)|\diff\pot^h|^2\,\d x
      + \int_\Omega \mob(\curv^h)\partial_h\diff\pot^h\cdot\diff\pot^h\,\d x \\
      &= -\frac12 \int_\Omega \partial_h\mob(\curv^h)|\diff\pot^h|^2\,\d x
      + \int_\Omega \big(\partial_h\mob(\curv^h)\diff\pot^h + \mob(\curv^h)\partial_h\diff\pot^h\big)\cdot\diff\pot^h\,\d x .
    \end{aligned}
  \end{equation}
  To simplify the second integral above,
  first observe that for every smooth function $\theta\in C^\infty(\overline\Omega)$
  it follows from \eqref{eq.Neumann2} that
  \begin{align}
    \label{eq.Neumann2weak}
    \int_\Omega \mob(\curv^h)\diff\pot^h\cdot\diff\theta\,\d x = \int_\Omega \partial_s\curv^h\theta\,\d x.
  \end{align}
  Taking the $h$-derivative yields
  \begin{align}
    \label{eq.dhds}
    \int_\Omega \big(\partial_h\mob(\curv^h)\diff\theta + \mob(\curv^h)\partial_h\diff\pot^h\big)\cdot\diff\theta\,\d x
    = \int_\Omega \partial_h\partial_s\curv^h\theta\,\d x.
  \end{align}
 If $\theta$ satisfies homogeneous Neumann boundary conditions,
  $\nml\cdot\diff\theta=0$ on $\partial\Omega$, \EEE
  we obtain from \eqref{eq.perturbedgeodesic}
  \begin{align}
    \label{eq.heatweak}
    \int_\Omega \partial_h\curv^h\vartheta\,\d x
    = s \int_\Omega \curv^h\Delta\vartheta\,\d x,
  \end{align}
  and the $s$-derivative amounts, in view of \eqref{eq.Neumann2weak}, to
  \begin{align}
    \label{eq.dsdh}
    \int_\Omega \partial_s\partial_h\curv^h\vartheta\,\d x
    = \int_\Omega \curv^h\Delta\vartheta\,\d x + s\int_\Omega \mob(\curv^h)\diff\pot^h\cdot\diff\Delta\vartheta\,\d x,
  \end{align}
  and thus allows to express the mixed derivative $\partial_h\partial_s\curv^h$ in \eqref{eq.dhds}.
  Using as test function $\theta=\pot^h$ \EEE
  in \eqref{eq.dhds} and \eqref{eq.dsdh},
  the integrals in \eqref{eq.actionderivative} become
  \begin{equation}\label{a}
  \begin{aligned}
    \frac12 \int_\Omega \partial_h \big[\mob(\curv^h)|\diff\pot^h|^2\big]\,\d x
    &= -\frac12 \int_\Omega \partial_h\mob(\curv^h)|\diff\pot^h|^2\,\d x \\
    &+ \int_\Omega \curv^h\Delta\pot^h\,\d x + s\int_\Omega \mob(\curv^h)\diff\pot^h\cdot\diff\Delta\pot^h\,\d x.
  \end{aligned}
  \end{equation}
  We evaluate the first integral on the right-hand side,
  \begin{align*}
    -&\frac12 \int_\Omega \partial_h\mob(\curv^h)|\diff\pot^h|^2\,\d x
    = -\frac12 \int_\Omega\mob'(\curv^h)s\Delta\curv^h|\diff\pot^h|^2\,\d x \\
    &=  \frac{s}2 \int_\Omega\diff(\mob'(\curv^h)|\diff\pot^h|^2)\cdot\diff\curv^h\,\d x\\
    &= \frac{s}2\int_\Omega\diff\big(\mob'(\curv^h)\Big)\cdot\diff\curv^h|\diff\pot^h|^2\,\d x
    + \frac{s}2\int_\Omega\mob'(\curv^h)\diff\curv^h\cdot\diff\big(|\diff\pot^h|^2\big)\,\d x\\
    & = \frac{s}2\int_\Omega\mob''(\curv^h)|\diff\curv^h|^2|\diff\pot^h|^2\,\d x
    - \frac{s}2\int_\Omega \mob(\curv^h)\Delta\big(|\diff\pot^h|^2\big)\,\d x \\
    & \qquad +\frac{s}2\int_{\partial\Omega} \mob(\curv^h)\diff\big(|\diff\pot^h|^2\big)\cdot\nml\,\d\HH^{d-1}.
  \end{align*}
  Using the last identity in \eqref{a}, taking into account the Bochner formula
  \begin{align*}
    - \frac12\Delta(|\diff\pot^h|^2) + \diff\Delta\pot^h\cdot\diff\pot^h = - \|\diff^2\pot^h\|^2 \le 0,
  \end{align*}
  and that $\diff\big(|\diff\pot^h|^2\big)\cdot\nml \leq 0$ on $\partial\Omega$ since $\Omega$ is convex,
  see \eqref{inconv},
  we find that
  \begin{align}
    \label{eq.entropyalonggeodesic}
       \frac12 \int_\Omega \partial_h \big[\mob(\curv^h)|\diff\pot^h|^2\big]\,\d x
       \leq - \int_\Omega \partial_s \big[ U(\curv^h)\big] \,\d x
       + \frac{s}2\int_\Omega\mob''(\curv^h)|\diff\curv^h|^2|\diff\pot^h|^2\,\d x.
  \end{align}
  \EEE
  By concavity of $\mob$, this proves \eqref{eq.eulerentropy}.
\end{proof}
The following Lemma provides the last missing piece for proving \eqref{eq.h2}
by means of the flow interchange Lemma \ref{lem.flowinterchange},
namely the dissipation of the energy $\altenergy$ along the heat flow \eqref{eq.heat}.
\begin{lemma}
  \label{lem.heatestimate}
  Let $\aux:\hopen\to H^1(\Omega)$ be a solution to \eqref{eq.heat}.
  If
  \begin{equation}\label{finitedissipation}
    \liminf_{s\downarrow0}\frac1s\big(\altenergy[\aux_s]-\altenergy[\aux_0]\big)>-\infty,
  \end{equation}
  then $\aux_0\in H^2(\Omega)$ and
  \begin{align}
    \label{eq.heatestimate}
    -\liminf_{s\downarrow0}\frac1s\big(\altenergy[\aux_s]-\altenergy[\aux_0]\big)
    \geq \frac12 \int_\Omega \big(\Delta\aux_0\big)^2\,\d x - C\big( \gGG\energyo+\altenergy[\aux_0] \big),
  \end{align}
  where the constant $C$ depends only on $\Mob(s_0)$, $|\Omega|$ and
  \GGG $G$. \EEE
\end{lemma}
\begin{proof}
  By classical parabolic theory,
  the solution $\aux$ to \eqref{eq.heat} is smooth, 
  and for every $0<s_0<s_1<\infty$ it satisfies $0<\inf_{x\in\Omega }\aux_s(x)\le \sup_{x\in\Omega} \aux_s(x)<M$ for
  $(x,s)\in\Omega\times(s_0,s_1)$. \EEE
  Thus $\altenergy[\aux_s]$ is continuously differentiable with
  \GGG respect to $s$ in $[s_0,s_1]$ \EEE
  with
  \begin{align}\label{ab1}
    \frac{\d}{\d s}\altenergy[\aux_s]
    = \int_\Omega \diff\aux_s\cdot\diff\Delta\aux_s\,\d x + \int_\Omega G'(\aux_s)\Delta\aux_s\,\d x
    = - \int_\Omega (\Delta\aux_s)^2\,dx - \int_\Omega G''(\aux_s)|\diff\aux_s|^2\,\d x,
  \end{align}
  where the last equality follows after integration by parts,
  using that the boundary condition $\nml\cdot\diff\aux_s=0$ is satisfied for any $s>0$.
  Taking into account
  \eqref{Pre1}
  \EEE
  the second integral can be estimated as follows,
  \begin{align*}
    - \int_\Omega G''(\aux_s)|\diff\aux_s|^2\,\d x
    \leq - \int_\Omega G_1''(\aux_s)|\diff\aux_s|^2\,\d x
    \leq C \Big( \int_\Omega |\diff\aux_s|^2\,\d x + \int_\Omega \frac{|\diff\aux_s|^2}{\mob(\aux_s)}\,\d x \Big).
  \end{align*}
  Recall \eqref{eq.mbelow} and the identity
  \begin{align*}
    \frac{|\diff\aux_s|^2}{\aux_s} = 4 \big| \diff\sqrt{\aux_s} \big|^2.
  \end{align*}
  In case that $M=+\infty$, we obtain
  \begin{align*}
    \int_\Omega \frac{|\diff\aux_s|^2}{\mob(\aux_s)}\,\d x
    &\leq \frac{s_0}{\mob(s_0)} \int_\Omega \frac{|\diff\aux_s|^2}{\aux_s}\,\d x
    + \frac{1}{\mob(s_0)} \int_\Omega |\diff \aux_s|^2\,\d x \\
    & \leq C\Big( \int_\Omega |\diff\sqrt{\aux_s}|^2\,\d x + \int_\Omega |\diff \aux_s|^2\,\d x \Big).
  \end{align*}
  Moreover, by H\"older's inequality and estimate \eqref{eq.villani} from the Appendix,
  \begin{align*}
    \int_\Omega |\diff\sqrt{\aux_s}|^2\,\d x
    \leq \Big( |\Omega| \int_\Omega |\diff\sqrt{\aux_s}|^4\,\d x \Big)^{1/2}
    \leq \Big( \frac{(d+8)|\Omega|}{16} \int_\Omega (\Delta\aux_s)^2\,\d x \Big)^{1/2} \leq \eps \int_\Omega (\Delta\aux_s)^2\,\d x + K_\eps .
  \end{align*}
  with $K_\eps=(d+8)|\Omega|/(64\eps)$.
  And analogously, if $M<\infty$,
  \begin{align*}
    \int_\Omega \frac{|\diff\aux_s|^2}{\mob(\aux_s)}\,\d x
    &\leq \frac{s_0}{\mob(s_0)} \int_\Omega \frac{|\diff\aux_s|^2}{\aux_s}\,\d x
    + \frac{M-s_0}{\mob(s_0)} \int_\Omega \frac{|\diff(M-\aux_s)|^2}{M-\aux_s}\,\d x \\
    &\leq C\Big( \int_\Omega |\diff\sqrt{\aux_s}|^2\,\d x + \int_\Omega |\diff\sqrt{M-\aux_s}|^2\,\d x \Big),
  \end{align*}
  where we use that
  \begin{align*}
    \int_\Omega |\diff\sqrt{M-\aux_s}|^2\,\d x
    \leq \eps \int_\Omega (\Delta\aux_s)^2\,\d x + K_\eps .
  \end{align*}
 Choosing $\eps$ above sufficiently small,
  and observing that
  \begin{align*}
    \int_\Omega |\diff\aux_s|^2\,\d x \leq \int_\Omega |\diff\aux_0|^2\,\d x,
  \end{align*}
  it thus can be achieved that
  \begin{align*}
    \frac{\d}{\d s}\altenergy[\aux_s] \leq - \frac12 \int_\Omega (\Delta\aux_s)^2\,\d x + C\big( 1 + \|\aux_0\|_{H^1(\Omega)}^2 \big),
  \end{align*}
  for a suitable constant $C$.

  \GGG Recall that the curve $s\mapsto \aux_s$ is continuous in $H^1(\Omega)$ 
  and that $G$ can be decomposed as in \eqref{eq:21}.
  The continuity of $G_\conc$ and the lower bound \eqref{eq:24} yield that
  \begin{displaymath}
    \lim_{s\downarrow0}\int_\Omega G_\conc(v_s)\,\d x=\int_\Omega G_\conc(v_0)\,\d x;
  \end{displaymath}
  on the other hand, since $G_\conv$ is convex, we have
  \begin{displaymath}
    \int_\Omega G_\conc(v_s)\,\d x\le \int_\Omega G_\conv(v_0)\,\d x
    \quad\foralltext s>0
  \end{displaymath}
  so that Fatou's Lemma and the continuity of $G$ yields
  \begin{displaymath}
    \lim_{s\downarrow0}\int_\Omega G_\conv(v_s)\,\d x=\int_\Omega G_\conv(v_0)\,\d x.
  \end{displaymath}
  \EEE
  Consequently, the function $s\mapsto \energy[\nu_s]$ is continuous at $s=0$
  and we have that
  \begin{align*}
    \frac1s \big( \altenergy[\aux_s] - \energy[\aux_0] \big)
    & \leq - \frac12 \frac1s \int_\Omega (\Delta\aux_{\theta(s)})^2\,\d x + C\big( 1 + \|\aux_0\|_{H^1(\Omega)}^2 \big),
  \end{align*}
  with $0<\theta(s)<s$.
  By \eqref{finitedissipation} it follows that
  the family $\{\Delta\aux_{\theta(s)}\}_{s\in(0,s_0)}$ for $s_0>0$ is weakly compact in $L^2(\Omega)$.
  Since $v_s$ converges to $v_0$ strongly in $H^1(\Omega)$ as $s\downarrow 0$,
  we have that $v_0\in H^2(\Omega)$ and
  \begin{align*}
    - \liminf_{s\downarrow0} \frac1s \big( \altenergy[\aux_s] - \altenergy[\aux_0] \big)
    & \geq \frac12 \liminf_{s\downarrow0}  \int_\Omega (\Delta\aux_{\theta(s)})^2\,\d x - C\big( 1 + \|\aux_0\|_{H^1(\Omega)}^2 \big) \\
    & \geq \frac12 \int_\Omega (\Delta\aux_0)^2\,\d x - C\big( 1 + \|\aux_0\|_{H^1(\Omega)}^2 \big) .
  \end{align*}
  Another application of the estimate \eqref{eq.boundbelow} finally provides \eqref{eq.heatestimate}.
\end{proof}
\begin{proof}[Proof of Proposition \ref{prp.strong}]
    By Lemma \ref{lem.entropy} we can apply the flow interchange Lemma \ref{lem.flowinterchange}
    with $\auxil=\entropy$.
 By Lemma \ref{lem.heatestimate} applied to $\nu_0=\mu_\tau^n$ we have that
  $\sol_\tau^n$ lies in $H^2(\Omega)$, and
  by \eqref{eq.heatestimate} and \eqref{eq.flowinterchange}, for any $n\in\setN$, it follows that
  \begin{equation}
    \label{eq.nosum}
    \frac\tau2 \int_\Omega \big(\Delta\sol_\tau^n\big)^2\,\d x \leq
    \entropy[\mu_\tau^{n-1}] - \entropy[\mu_\tau^n] +
    C\big(\gGG\energyo
    +\energy[\mu_\tau^n]\big)\tau .
  \end{equation}
  Here the constant $C$ is the same as in \eqref{eq.heatestimate},
  and does not depend on $\tau$, on $n$ or on the solution $\mu_\tau$.

  \GGG
  Let $T>0$ and $\tau\in(0,1)$ be given,
  and define $N\in\setN$ such that $(N-1)\tau<T\leq N\tau$. In view of \eqref{eq.entropyest} and $\energy[\mu_\tau^n]\leq\energy[\mu_0]$,
  summing \eqref{eq.nosum} from $n=1$ to $n=N$,
  we find that the interpolating function $\bar\sol_\tau$ satisfies
  \begin{align*}
    \|\Delta\bar\sol_\tau \|^2_{L^2(0,T;L^2(\Omega))}
    \leq  \tau \sum_{n=1}^N \int_\Omega
    \big(\Delta\sol_\tau^n\big)^2\,\d x
    \leq 2CT(\gGG\energyo+\energy[\mu_0]),
  \end{align*}
  \EEE
  which is obviously independent of $\tau\in(0,1)$.
  Combining this with \eqref{eq.boundbelow} and applying again \eqref{eq.laplace},
  we conclude that $\bar\sol_\tau$ remains uniformly bounded in $L^2(0,T;H^2(\Omega))$ as $\tau\downarrow0$,
  for any $T>0$:
  \begin{align}
    \label{eq.h2tau}
    \int_0^T \|\bar\sol_\tau(t)\|_{H^2(\Omega)}^2\,\d t \leq CT(\GGG\energyo+\energy[\mu_0]) < +\infty .
  \end{align}
  By \eqref{eq.h2tau} we have that, up to subsequences, $\bar\sol_{\tau_n}$
  converge weakly to $\sol$  in $L^2(0,T;H^2(\Omega))$ for every
  $T>0$.
  \GGG Since we have already seen in Proposition \ref{prp.discrete}
  that $\bar\sol_\tau$ pointwise converge weakly in $H^1(\Omega)$ and
  thus strongly in $L^2(\Omega)$ by Rellich's Theorem, the dominated
  convergence theorem shows that $\bar\sol_\tau$ converges strongly in
  $L^2(0,T;L^2(\Omega))$.
  \EEE
  From here, strong convergence $\bar\sol_{\tau_n}\to\sol$ in $L^2(0,T;H^1(\Omega))$
  follows by standard interpolation between the uniform bound \eqref{eq.h2tau}.
\end{proof}
\begin{corollary}
  \label{cor.mobpressure}
  In the setting of \S \ref{subsec:basic} we have for all $T>0$
  \begin{align}
   \label{eq.pressurestrong}
    P(\sol_{\tau_n}) &\to P(\sol) \qquad \text{strongly in
      $L^1(0,T;L^1(\Omega))$},
    \intertext{and, if $\Mob$ is also Lipschitz (as for \eqref{MobL})}
    \label{eq.mobstrong}
    \mob(\bar\sol_{\tau_n}) &\to \mob(\sol) \qquad \text{strongly in $L^2(0,T;H^1(\Omega))$}.
  \end{align}
\end{corollary}
\begin{proof}
  \GGG
  \eqref{eq.pressurestrong} is trivial when $M<\infty$. When
  $M=+\infty$,
  by Sobolev imbedding Theorem and the uniform estimates \eqref{eq.h2}
  and \eqref{eq.estimate1} we know that
  \begin{equation}
    \label{eq:26}
    \int_0^T \int_\Omega\Big( (\bar \sol_\tau)^q+|\bar
    \sol_\tau|\Big)\,\d x\,\d t\le C_T
  \end{equation}
  uniformly with respect to $\tau$. Since $\bar u_\tau$ (up to
  subsequence) converges strongly to $u$ in $L^1((0,T)\times\Omega)$,
  we deduce the same property for $P(\bar u)$ thanks to \eqref{Pre1}.

  \eqref{eq.mobstrong} is a standard consequence of the fact that
  $\Mob$ is Lipschitz and $C^1$.
\end{proof}

%

\subsection{Weak formulation}

The remaining section is devoted to prove the following Proposition stating that
the time-continuous limit $\sol$ obtained before
is a weak solution in the sense of \eqref{eq.weak}.
\begin{proposition}
  \label{prp.weak}
  Under the assumptions of Theorem \ref{thm.main},
  let $V$ be a spatial test function satisfying
  \begin{equation}\label{hp:V}
    V \in C^\infty(\overline\Omega), \qquad
    \diff V\cdot\nml=0 \text{ on }\partial\Omega,
  \end{equation}
  and a temporal test function $\psi\in C^\infty_c(0,+\infty)$ be given.
  Then
  \begin{align}
    \label{eq.weaker}
    -\int_0^{+\infty} \psi'(t)\potential[\mu(t)]\,\d t =
    \int_0^{+\infty} \psi(t)\nonlin[\sol(t),V] \,\d t,
  \end{align}
  where the nonlinear functional $\nonlin$ is given by
  \begin{align*}
    \nonlin[\sol,V] :=  -\int_\Omega \Delta\sol\dv\big(\mob(\sol)\diff
    V\big)\,\d x + \int_\Omega P(\sol)\Delta V \,\d x .
  \end{align*}
\end{proposition}
In the spirit of the ideas developed in \cite{JordKindOtto},
we would like to use the flow interchange Lemma \ref{lem.flowinterchange}
with $\auxil:=\potential$ the potential energy functional
$\potential:\admdens\to \R$ defined by
\begin{align*}
  \potential[\mu] := \int_\Omega V(x)\,u(x)\, \d x,
\end{align*}
with a test function $V$ satisfying \eqref{hp:V}.
As already mentioned in the introduction,
the functionals $\potential$ are --- unfortunately --- never $\lambda$-convex (for any $\lambda\in\R$)
along geodesics of the space $(\gGG\admdens,\wass)$,
unless the mobility $\mob$ is a linear function \cite[Section 2.3]{CarrLisiSavaSlep}.
To cure this problem, we shall construct a $\lambda_\eps$-flow for the regularized functional
\begin{align}\label{regpotential}
  \potentialeps[\mu] := \potential[\mu] + \eps \entropy[\mu] ,
\end{align}
with $\eps>0$ instead,
which amounts to solutions of the classical viscous approximation of \eqref{eq.claw1},
\begin{align}
  \label{eq.claw2}
  \partial_s \aux_s - \dv\big( \mob(\aux_s)\diff V \big) - \eps\Delta\aux_s = 0 \quad \mbox{ in $(0,+\infty)\times\Omega$ with }
  \diff\aux_s\cdot\nml = 0 \mbox{ on $(0,+\infty)\times\partial\Omega$}.
\end{align}
\begin{proposition}
  \label{prp.potential}
  Under the assumptions \eqref{Mob1}, \eqref{MobL} on $\Mob$, suppose that $V$ satisfies \eqref{hp:V}.
  Define the semigroup $\SG_\eps$ by taking $\SG_\eps^s\aux_0=\aux_s$,
  the unique solution to \eqref{eq.claw2} with initial condition $\aux_0$.
  Then $\SG_\eps$ extends to a $\lambda_\eps$-flow $\SG_\eps$ for $\potentialeps$ with respect to $\wass$,
  with some $\lambda_\eps\geq -K/\eps$ where $K>0$ only depends on $V$ and $\Mob$.
\end{proposition}
%
%
\begin{proof}
  As in the proof of Lemma \ref{lem.heatconvex},
  we need to verify \eqref{eq.eulerian} for the flow $\SG_\eps$ and the functional $\auxil=\potentialeps$.
  The calculations are similar to the proof there, but more terms need to be controlled.

  Below, we shall implicitly use various properties of the solution semi-group $\SG_\eps$ for \eqref{eq.claw2}.
  A summary of these relevant properties are given in Lemma \ref{lem.tootechnical} in the Appendix.
  In particular, note that $\SG_\eps$ is well-defined
  and $L^1$-continuous on the admissible densities $\admdens$,
  and that it leaves the regular densities $\regdens$ invariant.

  In order to define a family of mollifications $\{\potential_{\eps,n},\SG_{\eps,n}\}$ for $\potentialeps,\SG_\eps$,
  we assume without restriction that $0\in\Omega$,
  we take a monotone sequence $\eta_n\downarrow 1$
  and we define $\Omega_n:=\eta_n\Omega=\{\eta_n x:x\in\Omega\}$,
  choosing $\delta_n\downarrow0$ so that
  $\Omega_{[\delta_n]}\subset \Omega_n$.
  We define $V_n(x)=V(x /\eta_n)$ and
  $\potential_{\eps,n}[\sol]=\int_{\Omega_n}V_n(x)\d x + \eps
  \entropy[\sol]$.
  Then $V_n$ satisfies \eqref{hp:V} in $\Omega_n$
  We define for every $n$ the solution semi-group $\SG_{\eps,n}$ of the problem \eqref{eq.claw2} for $V_n$
  on the domain $\Omega_n$.
  It is not difficult to check that all the conditions of Definition \ref{def:mollification} are satisfied.


  We turn to prove \eqref{eq.eulerian}, writing for simplicity $\Omega$ in place of $\Omega_n$ everywhere.
  The $s$-derivative of $\potentialeps$ amounts to
  \begin{align*}
    - \int_\Omega \partial_s\big[ \eps U(\curv^h) + V\curv^h \big]\,\d x
    = \eps \int_\Omega \curv^h\Delta\pot^h\,\d x + \int_\Omega \mob(\curv^h)\diff V\cdot\diff\pot^h\,\d x.
  \end{align*}
  Moreover, the weak formulation \eqref{eq.heatweak} is modified as follows,
  \begin{align}
    \label{eq.clawweak}
    \int_\Omega \partial_h\curv^h\vartheta\,\d x = s\eps\int_\Omega\curv^h\Delta\vartheta\,\d x - s\int_\Omega\mob(\curv^h)\diff V\cdot\diff\vartheta\,\d x,
  \end{align}
  and, consequently, \eqref{eq.dsdh} is replaced by
  \begin{align*}
    \int_\Omega \partial_s\partial_h\curv^h\vartheta\,\d x
    & = \eps\int_\Omega\curv^h\Delta\vartheta\,\d x + s\eps \int_\Omega \mob(\curv^h)\diff\pot^h\cdot\diff\Delta\vartheta\,\d x \\
    & \quad - \int_\Omega\mob(\curv^h)\diff V\cdot\diff\vartheta\,\d x - s\int_\Omega\mob(\curv^h)\diff\pot^h\cdot\diff\big(\mob'(\curv^h)\diff V\cdot\diff\vartheta\big).
  \end{align*}
  Performing the same manipulations
  as in the proof of Lemma \ref{lem.heatconvex},
  one obtains
  \begin{align*}
    \frac12&\int_\Omega\partial_h\mob(\curv^h)|\diff\pot^h|^2\,\d x
    = \frac{\eps s}2\int_\Omega \mob'' (\curv^h)|\diff\curv^h|^2|\diff\pot^h|^2\,\d x
    - s\eps\int_\Omega\mob(\curv^h)\Delta\big(|\diff\curv^h|^2\big)\,\d x \\
    & \quad -s\int_\Omega \mob(\curv^h)\mob'(\curv^h)\diff V\cdot\diff^2\pot^h\cdot\diff\pot^h\,\d x
    -s\int_\Omega \mob(\curv^h)\mob''(\curv^h)\diff\curv^h\cdot\diff V |\diff\pot^h|^2\,\d x.
  \end{align*}
  Summing up everything provides
  \begin{align}
    \nonumber
    \frac12 &\int_\Omega \partial_h \big[\mob(\curv^h)|\diff\pot^h|^2\big]\,\d x
    = - \partial_s \bigg( \int_\Omega V\curv^h\,\d x+ \eps\int_\Omega U(\curv^h) \,\d x \bigg) \\
    \label{eq.killer}
    & \quad + \frac{s\eps}{2}\int_\Omega\mob''(\curv^h)|\diff\curv^h|^2|\diff\pot^h|^2\,\d x \\
    \label{boundaryterm}
    & \quad + \frac{s}2\int_{\partial\Omega} \mob(\curv^h)\diff\big(|\diff\pot^h|^2\big)\cdot\nml\,\d\HH^{d-1} \\
    \label{eq.victim1}
    & \quad - s\int_\Omega\mob(\curv^h)\mob''(\curv^h)(\diff\curv^h\cdot\diff\pot^h\diff\pot^h\cdot\diff V
    - \diff\curv^h\cdot\diff V|\diff\pot^h|^2)\,\d x \\
    \label{eq.victim2}
    & \quad - s\int_\Omega\mob(\curv^h)\mob'(\curv^h)\diff\pot^h\diff^2V\diff\pot^h\,\d x.
  \end{align}
  \EEE
  We need to show that the sum of the terms from \eqref{eq.killer} to \eqref{eq.victim2}
  are less than $-s\lambda_\eps\Act$ for a sufficiently small (negative) constant $\lambda_\eps$.
  The integral in \eqref{eq.victim2} is readily controlled by a multiple of $\Act$,
  recalling that $\mob$ has the Lipschitz property \eqref{MobL} and observing that
  \begin{align*}
    \big| \mob(\curv^h)\mob'(\curv^h)\diff\pot^h\diff^2V\diff\pot^h \big|
    \le \sup_s|\mob'(s)| \|V\|_{C^2(\Omega)}\ \mob(\curv^h)|\diff\pot^h|^2.
  \end{align*}
  In order to absorb the integral in \eqref{eq.victim1} into the (non-positive) integral in \eqref{eq.killer}
  and a multiple of $\Act$,
  we apply Young's inequality to the integrand and estimate
  \begin{align*}
    &\big| \mob(\curv^h)\mob''(\curv^h)(\diff\curv^h\cdot\diff\pot^h\diff\pot^h\cdot\diff V
    - \diff\curv^h\cdot\diff V|\diff\pot^h|^2)\big|
    \le 2\mob(\curv^h)|\mob''(\curv^h)||\diff\curv^h||\diff\pot^h|^2|\diff V| \\
    & \qquad \le \frac{\eps}2(-\mob''(\curv^h)) |\diff\curv^h|^2|\diff\pot^h|^2
    + \frac2{\eps}\big(-\mob''(\curv^h)\mob(\curv^h)^2\big) |\diff\pot^h|^2|\diff V|^2 \\
    &\qquad \le \frac{\eps}{2}(-\mob''(\curv^h)) |\diff\curv^h|^2|\diff\pot^h|^2
    + \frac{2}{\eps}\|V\|_{C^1(\Omega)}^2\sup_s(-\mob''(s)\mob(s))\ \mob(\curv^h)|\diff\pot^h|^2.
  \end{align*}
  Thus, defining, for every $n$
 \begin{align*}
    \lambda_{\eps,n} := -\sup_s|\mob'(s)| \|V_n\|_{C^2(\Omega_n)} - \frac2\eps\|V_n\|_{C^1(\Omega_n)}^2\sup_s(-\mob''(s)\mob(s)),
  \end{align*}
  and recalling that \eqref{boundaryterm} is non-positive for convexity of $\Omega$ we obtain
  \begin{equation}
  \begin{aligned}\label{euler}
     \frac12 \int_{\Omega_n} \partial_h \big[\mob(\curv^h)|\diff\pot^h|^2\big]\,\d x
    + s\lambda_{\eps,n} \int_{\Omega_n} \big[\mob(\curv^h)|\diff\pot^h|^2\big]\,\d x \\
    \leq - \partial_s \bigg( \int_{\Omega_n} V_n\curv^h\,\d x+ \eps\int_{\Omega_n} U(\curv^h) \,\d x \bigg).
  \end{aligned}
  \end{equation}
  Defining $\lambda_\eps := \inf_{n}\lambda_{\eps,n}>-\infty$
  (thanks to the uniform boundedness of all the derivatives of $V_n$)
  the principal estimate \eqref{eq.eulerian} follows from \eqref{euler}.
  \EEE
\end{proof}
The flow interchange estimate \eqref{eq.flowinterchange} is applicable.
To obtain a sensible a priori estimate,
we still need to express the dissipation term in \eqref{eq.flowinterchange}.
\begin{lemma}
  \label{lem.clawestimate}
  Let $\aux_s$ be as in Proposition \ref{prp.potential},
  and assume that $\nu_0\in\Dom(\energy)\cap H^2(\Omega)$.
  Then
  \begin{align}
    \label{eq.clawestimate}
    - \liminf_{s\downarrow0} \frac1s \big( \energy[\nu_s] - \energy[\nu_0] \big)
    \geq -\nonlin[\aux_0,V] + \eps \Big( \frac12 \int_\Omega
    (\Delta\aux_0)^2\,\d x - C(\gGG\energyo + \energy[\nu_0]) \Big) .
  \end{align}
\end{lemma}
\begin{proof}
  For $\delta>0$ sufficiently small, define approximations of $G$ by
  \begin{align*}
    G_\delta(s):= \begin{cases}
      G\Big(\delta+\frac{M-2\delta}{M}s\Big) & \text{if $M<\infty$}, \\
      G(\delta+s) &\text{if $M=+\infty$}.
    \end{cases}
  \end{align*}
  This regularizes the possible singularities of $G'(s)$ for $s\downarrow0$ and $s\uparrow M$.
  Denote by $\energy_\delta$ the energy functional associated to $G_\delta$ instead of $G$.

  For the following calculations, we need some properties of solutions to the problem \eqref{eq.claw2},
  which are summarized in Lemma \ref{lem.tootechnical} in the Appendix.
  By \eqref{eq.regularity1} we have that $s\mapsto\energy_\delta[\nu_s]$
  is absolutely continuous and we can calculate for almost every $s>0$ its derivative
    \begin{align*}
    \frac{\d}{\d s}\energy_\delta[\nu_s]
    & = \eps \Big( -\int_\Omega \Delta\aux_s\Delta\aux_s\,\d x
    + \int_\Omega G_\delta'(\aux_s)\Delta\aux_s\,\d x \Big) \\
    & \qquad - \int_\Omega \Delta\aux_s\dv(\mob(\aux_s)\diff V)\,\d x
    + \int_\Omega G_\delta'(\aux_s)\dv(\mob(\aux_s)\diff V)\,\d x \\
    & \leq - \eps \Big( \frac12 \int_\Omega (\Delta\aux_s)^2\,\d x
    - C(\gGG\energyo+\energy_\delta[\nu_0]) \Big) + \nonlin_\delta[\aux_s,V] .
  \end{align*}
  The last estimate is obtained by treating the term multiplied by $\eps$
  exactly as in the proof of Lemma \ref{lem.heatestimate},
  and integrating by parts in the last two integrals
  (which is allowed for the smooth approximation $G_\delta$ and does not produce boundary terms since $V$ satisfies homogeneous Neumann conditions \eqref{hp:V}).
  Moreover, following the proof of Lemma \ref{lem.heatestimate}, is is easy to check that the constant $C$ in the last integral can be chosen uniformly with respect to $\delta$.
  Then we have
  \begin{align*}
    \frac{\energy_\delta[\nu_s]-\energy_\delta[\nu_0]}{s}
    & \leq \frac{1}{s}\int_0^s  \Big( -\frac{\eps}{2} \int_\Omega (\Delta\aux_t)^2\,\d x
     + \nonlin_\delta[\aux_t,V] \Big)\,\d t + \eps C(1+\energy_\delta[\nu_0]).
  \end{align*}
    By \eqref{Pre1cons} it is easy to check that
    $\int_\Omega G_\delta (v)\,\d x \to \int_\Omega G (v)\,\d x$ and
    $\int_\Omega P_\delta (v)\Delta V\,\d x \to \int_\Omega P (v)\Delta V\,\d x$ as $\delta\downarrow 0$.
    Passing to the limit as $\delta\downarrow 0$ we obtain
    \begin{align*}
    \frac{\energy[\nu_s]-\energy[\nu_0]}{s}
    & \leq \frac{1}{s}\int_0^s   \Big(- \frac{\eps}{2} \int_\Omega (\Delta\aux_t)^2\,\d x
     + \nonlin[\aux_t,V] \Big)\,\d t + \eps C(1+\energy[\nu_0]).
  \end{align*}
    By the right continuity property \eqref{eq.continuity1} we can pass to the limit
    by $s\downarrow0$ obtaining \eqref{eq.clawestimate}.
\end{proof}
The flow interchange estimate \eqref{eq.flowinterchange} provides the following.
\begin{lemma}
  Let $V$ be a given test function
  satisfying \eqref{hp:V}
  and $\psi\in C^\infty_c(0,+\infty)$ be a given temporal test function
  satisfying $\psi\geq0$.
  Then,
  \begin{align}
    \label{eq.weakintermediate}
    - \int_0^\infty \psi'(t) \potentialeps[\bar\mu_\tau(t)]\,\d t
    \leq \int_0^\infty \bar\psi_\tau(t) \nonlin[\bar\sol_\tau(t),V]\,\d t + C\frac{\tau}{\eps},
  \end{align}
  where the simple function $\bar\psi_\tau:(0,+\infty)\to[0,+\infty)$
  is defined by $\bar\psi_\tau(t)=\psi((n-1)\tau)$ for $(n-1)\tau<t\leq n\tau$ for all $n\in\setN$.
  The constant $C$ in \eqref{eq.weakintermediate} is independent of $\tau$ and $\eps$ and depends only on the test functions $V$ and $\psi$,
  and on the initial energy $\energy[\mu_0]$.
\end{lemma}
\begin{proof}
  Since $\bar\mu_\tau$ is a simple function with respect to $t\geq0$,
  which is constant on intervals $((n-1)\tau,n\tau]$,
  and $\psi$ is smooth with compact support,
  it follows that, for some sufficiently large $N\in\setN$,
  \begin{equation}\label{cd}\begin{aligned}
    - \int_0^\infty \psi'(t) \potentialeps[\bar\mu_\tau(t)]\,\d t
    &= -\sum_{n=1}^N \int_{(n-1)\tau}^{n\tau} \psi'(t)\potentialeps[\mu_\tau^n]\,\d t \\
    &= \sum_{n=1}^N \big(\psi((n-1)\tau)-\psi(n\tau)\big)\,\potentialeps[\mu_\tau^n] \\
    &= \sum_{n=1}^N \psi((n-1)\tau) \big(\potentialeps[\mu_\tau^n]-\potentialeps[\mu_\tau^{n-1}]\big) .
  \end{aligned}\end{equation}
  By Proposition \ref{prp.potential} we can apply the flow interchange Lemma \ref{lem.flowinterchange} with $\auxil=\potentialeps$.
  By \eqref{eq.h2regularity} we can apply Lemma \ref{lem.clawestimate} with $\nu_0=\mu_\tau^n$.
  Combining the flow interchange estimate \eqref{eq.flowinterchange}
  and inequality \eqref{eq.clawestimate} we find
  \begin{equation}\label{cdd}
    \potentialeps[\mu_\tau^n] - \potentialeps[\mu_\tau^{n-1}]
    \leq \tau \nonlin[\sol_\tau^n,V]
    + \eps\tau\Big( C(\energyo+\energy[\mu_0]) - \frac12 \int_\Omega (\Delta\sol_\tau^n)^2\,\d x \Big)
    + \frac{K}{2\eps} \wass(\mu_\tau^n,\mu_\tau^{n-1})^2.
  \end{equation}
  Combining \eqref{cd} with \eqref{cdd} and recalling that $\psi\geq0$ we obtain
  \begin{align*}
    - \int_0^{+\infty} \psi'(t) \potentialeps[\bar\mu_\tau(t)]\,\d t
    &\leq \tau \sum_{n=1}^N \psi((n-1)\tau)\nonlin[\mu_\tau^n,V] + \eps C(\energyo+\energy[\mu_0])\tau\sum_{n=0}^{N-1} \psi(n\tau) \\
    & \qquad + \frac{K\tau}{2\eps}\sup_{t>0}\psi(t) \, \tau\sum_{n=1}^{+\infty} \Big(\frac{\wass(\mu_\tau^n,\mu_\tau^{n-1})}{\tau}\Big)^2 \\
    &\leq \int_0^{+\infty} \bar\psi_\tau(t)\nonlin[\bar\mu_\tau(t),V]\,\d t + C(\energyo+\energy[\mu_0])\eps\int_0^{+\infty} \bar\psi_\tau(t)\,\d t \\
    & \qquad + \frac{K\tau}{\eps}\sup_{t>0}\psi(t) (\energy[\mu_0] +\energyo) ,
  \end{align*}
  where the energy inequality \eqref{eq.estimate1} has been used to obtain the last line.
  The claim \eqref{eq.weakintermediate} follows.
\end{proof}
In order to finish the proof of \eqref{eq.weaker},
we pass to the time-continuous limit $\tau\downarrow0$ and the limit as $\eps\downarrow0$ simultaneously,
in such a way that the remainder term in \eqref{eq.weakintermediate} goes to zero.
\begin{proof}[Proof of Proposition \ref{prp.weak}]
  For definiteness,
  let $(\tau_n)_{n\in\setN}$ be a vanishing sequence for which $\bar\sol_{\tau_n}\to\sol$ strongly in $L^2(0,T,H^1(\Omega))$ according with Proposition \ref{prp.strong}.
  Without loss of generality, for \eqref{eq.strongconv} we may further assume
  that $\bar\sol_{\tau_n}\to\sol$ and $\diff\bar\sol_{\tau_n}\to\diff\sol$ almost everywhere on $(0,+\infty)\times\Omega$.
  It is sufficient to choose the vanishing sequence $\eps_n:=\sqrt{\tau_n}$
  in order to have that $C\frac{\tau_n}{\eps_n}\downarrow0$ in \eqref{eq.weakintermediate}.

  We start by proving convergence of the left-hand side in \eqref{eq.weakintermediate}.
  By the bounds from \eqref{eq.entropyest} and the the monotonicity of the energy \eqref{eq.estimate1}, one finds that
  \begin{align}\label{eq:1}
    \big| \potentialeps[\bar\mu_\tau(t)] - \potential[\bar\mu_\tau(t)] \big| \leq \eps C (\energyo+\energy[\mu_0])
  \end{align}
  for every $\tau>0$ and $t\geq0$.
  Choosing $T>0$ such that $\supp(\psi)\subset[0,T]$, using \eqref{eq:1}, we have
  \begin{equation*}
  \begin{aligned}
    & \left|\int_0^T\psi'(t) \potential_{\eps_n}[\bar\mu_{\tau_n}(t)]\,\d t
         -\int_0^T\psi'(t) \potential[\mu(t)]\,\d t \right|\\
    &\qquad\qquad \leq \sup_{t\in(0,T)}|\psi'(t)| \left( \int_0^T \big| \potential_{\eps_n}[\bar\mu_{\tau_n}(t)] - \potential[\bar\mu_{\tau_n}(t)] \big|\,\d t + \int_0^T \big| \potential[\bar\mu_{\tau_n}(t)] - \potential[\mu(t)] \big|\,\d t \right) \\
    &\qquad\qquad \leq \sup_{t\in(0,T)}|\psi'(t)| \left( T\eps_n C (\energyo+\energy[\mu_0])  + \sup_{x\in\Omega}|V(x)|\int_0^T    \int_\Omega |\bar\sol_{\tau_n}-\sol|\,\d x\,\d t \right)
  \end{aligned}
  \end{equation*}
  which shows that
  \begin{align}
    \label{eq.lim}
    \lim_{n\to+\infty} \int_0^{+\infty} \psi'(t) \potential_{\eps_n}[\bar\mu_{\tau_n}(t)]\,\d t
     =\int_0^{+\infty} \psi'(t) \potential[\bar\mu(t)]\,\d t.
  \end{align}
  From \eqref{eq.weakintermediate} and \eqref{eq.lim} one concludes that
  \begin{align}
    \label{eq.liminf}
    -\int_0^{+\infty} \psi'(t) \potential[\bar\mu(t)]\,\d t
    \leq \liminf_{n\to\infty} \int_0^{+\infty} \bar\psi_{\tau_n}(t) \nonlin[\bar\sol_{\tau_n}(t),V]\,\d t.
  \end{align}
  Next, we claim that the minimum limit in \eqref{eq.liminf} is actually a limit,
  and that
  \begin{align}
    \label{eq.strong1}
    \lim_{n\to+\infty} \int_{\Omega_T} \bar\psi_{\tau_n} \Delta\bar\sol_{\tau_n} \dv\big(\mob(\sol_{\tau_n})\diff V\big) \,\d x\,\d t
    &= \int_{\Omega_T} \psi\Delta\bar\sol\dv\big(\mob(\bar\sol)\diff V\big)\,\d x\,\d t, \\
    \label{eq.strong2}
    \lim_{n\to\infty} \int_{\Omega_T} \bar\psi_{\tau_n}P_i(\bar\sol_{\tau_n})\Delta V \,\d x\,\d t
    &= \int_{\Omega_T} \psi P_i(\sol)\Delta V \,\d x\,\d t ,
  \end{align}
  for $i=1,2$.
  In fact, \eqref{eq.strong1} and \eqref{eq.strong2} follow almost immediately from Corollary \ref{cor.mobpressure}:
  Combining \eqref{eq.mobstrong} with the weak convergence \eqref{eq.weakconv}
  and the uniform convergence of $\bar\psi_{\tau_n}$ to $\psi$ in $\Omega_T$, one obtains \eqref{eq.strong1}.
  And recalling that $\bar\psi_{\tau_n}$ uniformly converges to $\psi$ in $\Omega_T$,
  we obtain \eqref{eq.strong2} from \eqref{eq.pressurestrong}.

  Inserting \eqref{eq.strong1} and \eqref{eq.strong2} into \eqref{eq.liminf} we obtain that
  \begin{align}
    \label{eq.V}
    - \int_0^{+\infty} \psi'(t)\potential[\mu(t)]\,\d t \leq \int_0^{+\infty} \psi(t)\nonlin[\sol(t),V]\,\d t
  \end{align}
  for all $V \in C^\infty(\overline\Omega)$ satisfying \eqref{hp:V}, and all non-negative $\psi\in C^\infty_c(0,+\infty)$.
  Exchanging $V$ with $-V$ in \eqref{eq.V} yields the respective equality \eqref{eq.weaker}.
  Trivially,
  \eqref{eq.weaker} extends from non-negative test functions $\psi$ to \emph{all} $\psi\in C^\infty_c(0,+\infty)$,
  thus finishing the proof.
\end{proof}
Since any space-temporal test function in \eqref{eq.weak}
$\zeta\in C^\infty((0,+\infty)\times\overline\Omega)$,
with $\diff \zeta\cdot\nml=0$ on $\partial\Omega$,
can be approximated in $C^\infty((0,+\infty)\times\overline\Omega)$
by sums of functions of the type $\zeta(t,x)=\psi(t)V(x)$
with $V \in C^\infty(\overline\Omega)$ satisfying \eqref{hp:V}
and $\psi\in C^\infty_c(0,+\infty)$,
Theorem \ref{thm.main} follows.

\section{Proof of Theorem \ref{thm.main2}}
\label{sct.general}
\subsection{Approximation}
Theorem \ref{thm.main2} is now proven by approximation of the more general mobility function $\mob$ satisfying \eqref{Mob4}
by mobilities $\mob_\delta$ that have the Lipschitz property \eqref{MobL}.
To this end, define for all $\delta>0$ sufficiently small:
\begin{itemize}
\item if $M<+\infty$:
  \begin{align*}
    \mob_\delta(s) := \mob\Big(\frac{s^2_\delta-s^1_\delta}{M}s+s^1_\delta\Big)-\delta,
  \end{align*}
  where $s^1_\delta < s^2_\delta$ are the two solutions of $\mob(s)=\delta$.
\item if $M=+\infty$:
  \begin{align*}
    \mob_\delta(s) := \mob(s+s_\delta)-\delta.
  \end{align*}
  where $s_\delta>0$ is the unique solution of $\mob(s)=\delta$.
\end{itemize}
Introduce accordingly $P_\delta$ by
\begin{align}
  \label{eq:approxP}
  P_\delta(s) = \int_0^s \mob_\delta(r)G''(r)\d r.
\end{align}
%
\begin{lemma}
  \label{lem:thisisfartootechnical}
  For all $\delta>0$ sufficiently small, the $\mob_\delta$ are smooth functions 
  that have the Lipschitz property \eqref{MobL}
  and satisfy the pointwise bounds $0\le\mob_\delta\le\mob$.
  In particular, we have
  \begin{equation}
    \label{eq:27}
    \wass(u,v)\le \wassdelta(u,v)\qquad
    \foralltext u,v\in \admdens.
  \end{equation}
  For $\delta\downarrow0$, the $\mob_\delta$ converge monotonically and globally uniformly to $\mob$.
  Moreover, if $G$ satisfies \eqref{Pre1} with respect to $\mob$,
  then it
  also satisfies \eqref{Pre1} with respect to each $\mob_\delta$.
  Finally, the $P_\delta$ are continuous functions, and there is a constant $K$ such that
  \begin{align}
    \label{eq:Pdeltabound}
    -K(1+s^2)\le P_\delta(s) \le P(s) + K(1+s)
  \end{align}
  for all $s\in(0,M)$ and all $\delta>0$ sufficiently small,
  and $P_\delta$ converges to $P$ as $\delta\downarrow0$, uniformly on $[0,M]$ if $M<\infty$, or uniformly on each $[0,\bar s]$ if $M=+\infty$.
\end{lemma}
\begin{proof}
  Smoothness, non-negativity and the Lipschitz property of $\mob_\delta$ are evident from its definition, and the concavity and smoothness of $\mob$.
  In the case $M=+\infty$, also the upper bound $\mob_\delta\le\mob$ is a trivial consequence of concavity,
  as is the uniform convergence for $\delta\downarrow0$:
  \begin{align}
    \label{eq:mobdeltaconcavity}
    0 \le \mob(s)-\mob_\delta(s) \le \delta-\mob'(s+s_\delta)s_\delta \le \delta.
  \end{align}
  In the case $M<+\infty$, the upper bound can be proven as follows:
  assume that $\mob$ attains its maximal value at $\sigma\in(0,M)$;
  then $\mob_\delta$ attains its maximum at $\sigma_\delta=(\sigma-s^1_\delta)M/(s^2_\delta-s^1_\delta)$.
  Without loss of generality, assume $\sigma_\delta\le\sigma$.
  For all $s\in[0,\sigma_\delta]$, we have
  \begin{align*}
    \mob_\delta'(s) = \underbrace{\frac{s^2_\delta-s^1_\delta}{M}}_{\le1}
    \mob'\Big(\underbrace{\frac{s^2_\delta-s^1_\delta}{M}s+s^1_\delta}_{\ge s}\Big) \le \mob'(s)
  \end{align*}
  and thus also $\mob_\delta(s)\le\mob(s)$.
  A similar argument provides $\mob_\delta(s)\le\mob(s)$ for all $s\in[\sigma,M]$.
  For $s\in[\sigma_\delta,\sigma]$, 
  the inequality follows since $\mob_\delta$ is non-increasing and $\mob$ is non-decreasing on that interval.
  The argument for uniform convergence of $\mob_\delta$ to $\mob$ 
  is established essentially with the same argument as in \eqref{eq:mobdeltaconcavity},
  making again a case distinction whether $s\in[0,\sigma_\delta]$, $s\in[s_\delta,\sigma]$, or $s\in[\sigma,M]$.

  Condition \eqref{Pre1} on $G$ is less stringent for the approximations $\mob_\delta$ 
  since $\mob_\delta\le\mob$ in case $M=+\infty$, and $\mob_\delta/(1+\mob_\delta)\le\mob/(1+\mob)$ in case $M<\infty$.
  Concerning the continuity of $P_\delta$, we remark that, in view of $\mob_\delta\le\mob$,
  the integrability of $\mob(s)G''(s)$ near $s=0$ (and near $s=M$ if $M<\infty$) implies the respective integrability of $\mob_\delta(s)G''(s)$.
  Moreover,
  \begin{align*}
    \sup_{0<s<\bar s} |P(s)-P_\delta(s)| \le \int_0^{\bar s} \big(\mob(s)-\mob_\delta(s)\big)|G''(s)|\d s
  \end{align*}
  in combination with the pointwise convergence of $\mob_\delta$ to $\mob$ 
  implies uniform convergence of $P_\delta$ to $P$ on all intervals $[0,\bar s]$;
  notice that the dominated convergence theorem is applicable since $0\le(\mob-\mob_\delta)|G''|\le \mob |G''|$,
  and the latter is integrable by assumption.
  Finally, if $M=+\infty$, then \eqref{eq:Pdeltabound} is another consequence of \eqref{Pre1}.
  Indeed, on one hand,
  \begin{align*}
    P_\delta(s) = \int_0^s \mob_\delta(r)G''(r)\d r \ge -C \int_0^s(1+\mob(r))\d r \ge -K(1+s^2),
  \end{align*}
  and on the other hand, using also \eqref{eq:mobdeltaconcavity},
  \begin{align*}
    P_\delta(s) = P(s) - \int_0^s \big(\mob(r)-\mob_\delta(r)\big)G''(r)\d r
    &\le P(s) + \delta\int_0^s \big( G''(r)\big)_-\d r \\
    &\le P(s) + C\delta \int_0^s \Big( 1+\frac1{\mob(r)}\Big)\d r \le P(s)+K(1+s).
  \end{align*}
  For $M<\infty$, \eqref{eq:Pdeltabound} simply amounts to $\delta$-uniform boundedness of $P_\delta$,
  which is clear from the uniform convergence to $P$.
\end{proof}
\nodaniel

\subsection{Weak and strong convergence}
%
%
Lemma \ref{lem:thisisfartootechnical} implies that Theorem \ref{thm.main} is applicable 
to the approximate mobilities $\mob_\delta$ for each $\delta$ sufficiently small:
there exist respective solutions $\sol_\delta:[0,+\infty)\to\admdens$ of problem \eqref{eq.formalflow}--\eqref{eq.ic},
i.e.,
\begin{align}
  \label{eq.weakdelta}
  \int_0^{+\infty} \int_\Omega \partial_t\zeta\,\sol_\delta\,\d x\,\d t
  = \int_0^{+\infty}\int_\Omega\Delta\sol_\delta\,\diff\big(\mob_\delta(\sol_\delta)\diff\zeta\big)\,\d x\,\d t
  + \int_0^{+\infty}\int_\Omega P_\delta(\sol_\delta)\Delta\zeta\,\d x\,\d t
\end{align}
for all test functions $\zeta\in C^\infty_c((0,+\infty)\times\overline\Omega)$  such that $\diff \zeta\cdot\nml=0$ on $\partial\Omega$.
We wish to pass to the limit as $\delta\downarrow0$ in \eqref{eq.weakdelta}.
\begin{lemma}
  \label{lem:technical1}
  There exists an absolutely continuous curve $\sol:\hopen\to\admdens$ 
  and a vanishing sequence $\delta_k$ such that the $\sol_{\delta_k}$ converge to $\sol$ 
  weakly in $H^1(\Omega)$ and strongly in $L^2(\Omega)$ pointwise in time,
  as well as strongly in $L^2(0,T;L^2(\Omega))$ for every $T>0$ 
  and pointwise a.e.\ on $(0,+\infty)\times\Omega$.
\end{lemma}
\begin{proof}
  The curves $\sol_\delta$ satisfy estimate \eqref{eq.estimate2} in the respective metric $\wassdelta$ 
  with the global H\"older constant determined by $\energy[\mu_0]$.
  In view of \eqref{eq:27}, the family $(\sol_\delta)_{\delta>0}$ is equi-continuous with respect to $\wass$.
  Moreover, the sublevel sets of the energy $\energy$ are compact.
  The claim on convergence now follows by the variant of the Arzel\'a-Ascoli theorem given in \cite[Proposition 3.3.1]{AmbrGiglSava}.
  An application of the dominated convergence theorem with respect to time provides
  strong convergence of $\sol_\delta$ to $\sol$ also in sense of $L^2(0,T;L^2(\Omega))$,
  and thus (without loss of generality) also pointwise a.e.\ convergence.
\end{proof}
In the following we write $\delta\downarrow0$ to indicate 
``along a suitable vanishing sequence $\delta_k$ for $k\to\infty$''.

\begin{lemma}
  \label{lem:technical3}
  For every $T>0$,
  \begin{align}\label{eq.h2est}
    \limsup_{\delta\downarrow0} \int_0^T \|\sol_\delta(t)\|_{H^2(\Omega)}^2 \,\d t < +\infty.
  \end{align}
  Consequently,
  $\sol_\delta$ converges to $\sol$ weakly in $L^2(0,T;H^2(\Omega))$ and strongly $L^2(0,T;H^1(\Omega))$
  as $\delta\downarrow0$.
  Moreover, $P_\delta(\sol_\delta)$ converges to $P(\sol)$ in $L^1(0,T;L^1(\Omega))$.
\end{lemma}
\begin{proof}
  Define the $\delta$-approximations of the entropy functional $\entropy_\delta$ as in \eqref{eq.entropy}
  with $\Mob_\delta$ instead of $\Mob$.
  Then estimate \eqref{eq.entropyest} holds with a constant $C$ independent of $\delta$ for every $\entropy_\delta$,
  at least for all $\delta>0$ sufficiently small.
  Indeed, observe that $\Mob_\delta(s_0)\geq \frac12\Mob(s_0)$ if $\delta$ is small enough,
  and hence inequality \eqref{eq.mbelow} follows.
  In the same way, inequalities \eqref{eq.heatestimate} and \eqref{eq.h2tau}
  can be obtained with $\delta$-independent constants $C$.
  In combination, \eqref{eq.h2est} follows.

  The stated weak convergence of $\sol_\delta$ is now a consequence of Alaoglu's theorem and the uniqueness of the weak limit.
  The strong convergence is obtained by interpolation of the strong convergence in $L^2(0,T;L^2(\Omega))$ with the bound \eqref{eq.h2est}.
  To prove convergence of $P_\delta(\sol_\delta)$, we argue as in Corollary \ref{cor.mobpressure}.
\end{proof}
%

\subsection{Convergence of the mobility gradient}
%
The next goal is to establish convergence of $\mob_\delta(\sol_\delta)$ in $L^2(0,T;H^1(\Omega))$.
\begin{lemma}
  $\mob_\delta(\sol_\delta)$ converges to $\mob(\sol)$ strongly in $L^2(\Omega_T)$ for every $T>0$ as $\delta\downarrow0$.
\end{lemma}
\begin{proof}
  The uniform convergence of the mobility functions $\mob_\delta$ to $\mob$,
  and the pointwise a.e.\ convergence of $\sol_\delta$ to $\sol$ suffice to conclude
  pointwise a.e.\ convergence of $\mob_\delta(\sol_\delta)$ to $\mob(\sol)$ on $\Omega_T$.
  Moreover, if $M<\infty$, then $\Mob_\delta(\sol_\delta)$ is $\delta$-independently bounded,
  and by dominated convergence it follows that $\mob_\delta(\sol_\delta)$ converges strongly to $\mob(\sol)$ in $L^2(\Omega_T)$.
  In the case $M=+\infty$, the $\delta$-uniform bound $\Mob_\delta(s)\leq\Mob(s)\leq C(1+s)$
  and the strong convergence of $\sol_\delta$ in $L^2(\Omega_T)$
  imply equi-integrablity of $|\mob_\delta(\sol_\delta)|^2$ in $\Omega_T$.
  We invoke Vitali's theorem to conclude the proof.
\end{proof}
For the proof of convergence of the gradients $\diff\mob_\delta(\sol_\delta)$ in $L^2(\Omega_T)$,
we distinguish the cases $M<\infty$ and $M=\infty$.
%
\begin{lemma}
  \label{lem:supertechnical}
  Assume $M<\infty$.
  Define $g:[0,M]\to\setR$ by $g_\delta(s):=\sqrt{s(M-s)}\mob_\delta'(s)$ for all $\delta>0$ sufficiently small.
  For $\delta\downarrow0$, the $g_\delta$ converge uniformly 
  to the continuous function $g_0:[0,M]\to\setR$ given by $g_0(s)=\sqrt{s(M-s)}\mob'(s)$ for all $s\in(0,M)$, and $g_0(0)=g_0(M)=0$.
\end{lemma}
\begin{proof}
  Let $\epsilon>0$ be given.
  Since $\sqrt{s(M-s)}\mob'(s)\to0$ for $s\downarrow0$ and for $s\uparrow M$, respectively, by hypothesis \eqref{Mob4},
  and since $\mob'(s)$ is an non-increasing function on $(0,M)$,
  there are $\sigma^1_\epsilon<\sigma^2_\epsilon$ such that $|g_\delta(s)|\le\sqrt{s(M-s)}|\mob'(s)|<\epsilon$ for all $s\in[0,\sigma^1_\epsilon]\cup[\sigma^2_\epsilon,M]$,
  and all $\delta>0$ sufficiently small.
  For $\delta\downarrow0$, the points $s^1_\delta$ and $s^2_\delta$ in the definition of $\mob_\delta$ converge to $0$ and to $M$, respectively.
  By smoothness of $\mob$, one thus has local uniform convergence of $\mob_\delta'$ to $\mob'$,
  and consequently $|g_\delta(s)-g_0(s)|<\epsilon$ for all $s\in[\sigma^1_\epsilon,\sigma^2_\epsilon]$, for sufficiently small $\delta>0$.
  The uniform convergence of the $g_\delta$ also proves continuity of $g_0$.
\end{proof}
\nodaniel
\begin{lemma}
  \label{lem:technical2a}
  Assume $M<\infty$.
  Then $\diff\mob_\delta(\sol_\delta)$ converges to $\diff\mob(\sol)$ strongly in $L^2(\Omega_T)$ for every $T>0$ as $\delta\downarrow0$.
\end{lemma}
\begin{proof}
  Observe that
  \begin{align}
    \label{eq.rootA}
    |\diff\mob_\delta(\sol_\delta(t))|^2 =  (g_\delta(\sol_\delta))^2\frac{1}{M}
    \big(|\diff\sqrt{\sol_\delta(t)}|^2+|\diff\sqrt{M-\sol_\delta(t)}|^2\big)
  \end{align}
  for every $t\geq0$ at which $\sol_\delta(t)\in H^2(\Omega)$.
  Pointwise convergence of $\sol_\delta$ to $\sol$ almost everywhere on $\Omega_T$
  and uniform convergence of $g_\delta$ to $g_0$ imply
  pointwise convergence of the compositions $g_\delta(\sol_\delta)$ to $g_0(\sol)$ almost everywhere.
  In combination with the $\delta$-uniform boundedness of $g_\delta$
  it follows in particular that $g_\delta(\sol_\delta)\to g_0(\sol)$ in $L^4(\Omega_T)$.

  Now let $Q\subset\Omega_T$ be a measurable subset.
  From \eqref{eq.rootA} we obtain
  \begin{align}
    \nonumber
    \int_Q |\diff\mob_\delta(\sol_\delta)|^2\,\d x\,\d t
    & = \frac{1}{M}\int_Q \big|g_\delta(\sol_\delta)\big|^2 \big|\diff\sqrt{\sol_\delta}\big|^2 \,\d x\,\d t
    +\frac{1}{M}\int_Q \big|g_\delta(\sol_\delta)\big|^2 \big|\diff\sqrt{M-\sol_\delta}\big|^2 \,\d x\,\d t\\
    \label{eq.diffmob1}
    & \leq \frac{1}{M} \bigg( \int_0^T \|\diff\sqrt{\sol_\delta(t)}\|_{L^4}^4 \,\d t \bigg)^{1/2}
    \bigg( \int_Q g_\delta(\sol_\delta)^4\,\d x\,\d t \bigg)^{1/2} \\
    \nonumber & + \frac{1}{M} \bigg( \int_0^T \|\diff\sqrt{M-\sol_\delta(t)}\|_{L^4}^4 \,\d t \bigg)^{1/2}
    \bigg( \int_Q g_\delta(\sol_\delta)^4\,\d x\,\d t \bigg)^{1/2}.
  \end{align}
  %
  Choose $Q=Z_T:=\{(t,x)\in\Omega_T:\sol(t,x)=0\text{ or }\sol(t,x)=M\}$.
  Since $g_\delta(u_\delta)\to g_0(u)=0$ in $L^4(Z_T)$,
  the right-hand side of \eqref{eq.diffmob1} vanishes as $\delta\downarrow0$,
  and $\diff\mob_\delta(\sol_\delta)\to0$ in $L^2(Z_T)$.

  Further, since the $Q$ in inequality \eqref{eq.diffmob1} can be chosen as any arbitrary subset of $\Omega_T\setminus Z_T$,
  the equi-integrability of $|g_\delta(\sol_\delta)|^4$ on $\Omega_T$ is inherited to $|\diff\mob_\delta(\sol_\delta)|^2$ on $\Omega_T\setminus Z_T$.
  On $\Omega_T\setminus Z_T$, the composition $\mob'(\sol)$ is everywhere finite,
  and $\mob_\delta'(\sol_\delta)\to\mob'(\sol)$ pointwise a.e.
  In combination with the pointwise a.e.\ convergence of $\diff\sol_\delta$,
  it follows that
  \begin{align*}
    \diff\mob_\delta(\sol_\delta) = \mob_\delta'(\sol_\delta)\diff\sol_\delta \to \mob'(\sol)\diff\sol=\diff\mob(\sol)
  \end{align*}
  strongly in $L^2(\Omega_T\setminus Z_T)$, invoking Vitali's theorem once again.

  The proof of strong convergence $\mob_\delta(\sol_\delta)\to\mob(\sol)$ in $L^2(0,T;H^1(\Omega))$
  is now concluded by observing the following.
  Since $\mob_\delta(\sol_\delta)\to\mob(\sol)$ in $L^2(\Omega_T)$,
  the strong $L^2(\Omega_T)$-limit of $\diff\mob_\delta(\sol_\delta)$ coincides
  with the (uniquely determined) distributional derivative $\diff\mob(\sol)$.
  In particular, $\diff\mob(\sol)$ vanishes a.e. on the set $Z_T$.
\end{proof}
\begin{lemma}
  \label{lem:technical2b}
  Assume $M=+\infty$.
  Then $\diff\mob_\delta(\sol_\delta)$ converges to $\diff\mob(\sol)$ strongly in $L^2(\Omega_T)$ for every $T>0$ as $\delta\downarrow0$.  
\end{lemma}
\begin{proof}
  Similar as in the proof of Lemma \ref{lem:technical2a},
  define $g_\delta:\hopen\to\setR$ by $g_\delta(s):=\sqrt{s}\mob_\delta'(s)$ for all $\delta>0$ small enough.
  Analogously to Lemma \ref{lem:supertechnical},
  one obtains convergence --- uniformly on every interval $[0;\bar s]$ --- of $g_\delta$ to the continuous function $g_0:\hopen\to\setR$
  with $g_0(s):=\sqrt{s}\mob'(s)$ for $s\in(0,+\infty)$, and $g_0(0)=0$.
  Like in \eqref{eq.rootA}, we observe that
  \begin{align}
    \label{eq.rootB}
    |\diff\mob_\delta(\sol_\delta(t))|^2 =  4(g_\delta(\sol_\delta))^2|\diff\sqrt{\sol_\delta(t)}|^2
  \end{align}
  for every $t\geq0$ at which $\sol_\delta(t)\in H^2(\Omega)$.
  As before, we conclude that $g_\delta(\sol_\delta)$ converges to $g_0(\sol)$ almost everywhere on $(0,+\infty)\times\Omega$.
  Moreover, by construction of the $\mob_\delta$ and \eqref{Mob4},
  one has $0\leq g_\delta(s)\leq C(1+s^{1/2})$ for all $s\geq0$.
  Since $|\sol_\delta|^2$ is equi-integrable in $\Omega_T$ for arbitrary $T>0$,
  the compositions $|g_\delta(\sol_\delta)|^4$ are also equi-integrable in $\Omega_T$.
  By Vitali's Theorem, it follows that $g_\delta(\sol_\delta)\to g_0(\sol)$ in $L^4(\Omega_T)$.
  
  From \eqref{eq.rootB},
  we conclude that
  \begin{align}
    \nonumber
    \int_Q |\diff\mob_\delta(\sol_\delta)|^2\,\d x\,\d t
    & \leq 4 \int_Q \big|g_\delta(\sol_\delta)\big|^2 \big|\diff\sqrt{\sol_\delta}\big|^2 \,\d x\,\d t \\
    \label{eq.diffmob}
    & \leq 4 \bigg( \int_0^T \|\diff\sqrt{\sol_\delta(t)}\|_{L^4}^4 \,\d t \bigg)^{1/2}
    \bigg( \int_Q g_\delta(\sol_\delta)^4\,\d x\,\d t \bigg)^{1/2}
  \end{align}
  holds for any measurable set $Q\subset\Omega_T$.
  From this point on, the proof is identical to the one for Lemma \ref{lem:technical2a},
  with the only change that $Z_T:=\{(t,x)\in\Omega_T:\sol(t,x)=0\}$.
\end{proof}

\subsection{End of the proof}
We are now able to pass to the limit in the weak formulation \eqref{eq.weakdelta}.
By Lemma \ref{lem:technical1}, 
$\Delta\sol_\delta$ converges to $\Delta\sol$ weakly in $L^2(\Omega_T)$ as $\delta\downarrow0$,
and by Lemma \ref{lem:technical2a} or \ref{lem:technical2b}, respectively, 
$\diff\mob_\delta(\sol_\delta)$ converges to $\diff\mob(\sol)$ strongly in that space.
Hence, the term inside the first integral on the right-hand side of \eqref{eq.weakdelta} 
converges weakly to the desired limit in $L^1(\Omega_T)$.
Convergence of the second integral follows from the strong convergence of $P_\delta(\sol_\delta)$ to $P(\sol)$ in $L^1(\Omega_T)$
stated in Lemma \ref{lem:technical3}.
This finishes the proof of Theorem \ref{thm.main2}.
\nodaniel

\begin{appendix}
  \section{Appendix}
  We recall a Sobolev like inequality, that will be useful in order to
  estimate the rate of dissipation of the energy $\energy$ along the heat flow.
  \begin{lemma}\label{lem.Sobolev}
    Assume that $\Omega$ is a smooth \EEE convex open set
    and that $u\in H^2(\Omega)$ satisfies homogeneous Neumann boundary conditions, $\nml\cdot\diff u=0$ on $\partial\Omega$.
    Then
    \begin{align}
      \label{eq.laplace}
      \int_\Omega \|\diff^2u\|^2\,\d x \leq \int_\Omega (\Delta u)^2\,\d x \leq d \int_\Omega \|\diff^2u\|^2\,\d x.
    \end{align}
    If, in addition, $u$ is non-negative,
    then $\sqrt{u}\in W^{1,4}(\Omega)$ and
    \begin{align}
      \label{eq.villani}
      16 \int_\Omega \big|\diff\sqrt{u}\big|^4\,\d x \leq (d+8) \int_\Omega (\Delta u)^2\,\d x .
    \end{align}
  \end{lemma}
  \begin{proof}
    By density, it obviously suffices to prove \eqref{eq.laplace}
    for $u\in C^\infty(\overline\Omega)$ satisfying  $\nml\cdot\diff u=0$ on $\partial\Omega$.
    Integrating by parts, it follows
    \begin{align*}
      \int_\Omega \|\diff^2u\|^2\,\d x &= \int_{\partial\Omega} \nml\cdot\diff^2u\cdot\diff u\,\d \sigma - \int_\Omega \diff\Delta u\cdot\diff u\,\d x \\
      &= \int_{\partial\Omega} \nml\cdot\diff^2u\cdot\diff u\,\d \sigma - \int_{\partial\Omega} \nml\cdot\diff u \Delta u\,\d \sigma + \int_\Omega (\Delta u)^2\,\d x .
    \end{align*}
    Now observe that the second boundary integral vanishes because of the no-flux boundary conditions,
    whereas the integrand in the first is pointwise non-negative,
    \begin{align}\label{inconv}
      \nml\cdot\diff^2u\cdot\diff u \leq 0.
    \end{align}
    (For a proof of this classical fact, see e.g. \cite[Lemma 5.2]{GianSavaTosc}.)
    This shows the first inequality in \eqref{eq.laplace}.
    The second inequality follows by observing the pointwise relation
    \begin{align*}
      (\Delta u)^2 = \sum_{i,j=1}^d \partial_{ii}u\partial_{jj}u \leq \frac12 \sum_{i,j=1}^d \big( \partial_{ii}u \big)^2 + \big(\partial_{jj}u \big)^2
      = d \sum_{i=1}^d \big( \partial_{ii}u \big)^2 \leq d \|\diff^2u\|^2 .
    \end{align*}
    The estimate \eqref{eq.villani} follows by combination of \cite[Lemma 3.1]{GianSavaTosc}
    with \eqref{eq.laplace} above.
  \end{proof}
  The next lemma summarizes selected properties of solutions to the viscous conservation law \eqref{eq.claw2}.
  These are used in the calculation of the right derivative of the energy $\energy$ along the corresponding flow.
  \begin{lemma}
    \label{lem.tootechnical}
    Assume that $\Omega$ is a smooth bounded domain,
    $\Mob$ satisfies \eqref{Mob1}, \eqref{MobL} and $V$ satisfies \eqref{hp:V}.
\EEE
 If $\aux_0\in\regdens$, then there exists a unique smooth classical solution 
    $\aux(s,x)$ of problem \eqref{eq.claw2} and $\aux(s,\cdot)\in\regdens$.
    If $\aux_0$ and $\bar\aux_0$ are two initial conditions in $\regdens$, then the $L^1$-contractivity estimate
    \begin{align*}
      \| \aux_s - \bar\aux_s \|_{L^1(\Omega)} \le \|\aux_0-\bar\aux_0\|_{L^1(\Omega)}
    \end{align*}
    holds at any time $s\ge0$ for the associated solutions $\aux$, $\bar\aux$. 
    If $\aux_0\in H^2(\Omega)$, then $\aux_s\in H^2(\Omega)$ for every $s\geq0$,
    \begin{equation}\label{eq.regularity1}
      \Delta \aux, \, \dv(\Mob(\aux)\diff V)\in L^\infty(0,+\infty;L^2(\Omega)),
    \end{equation}
    and the maps
    \begin{equation}\label{eq.continuity1}
      s\mapsto\Delta \aux_s, \quad  s\mapsto\dv(\Mob(\aux_s)\diff V) \quad \text{are right continuous with values in } L^2(\Omega).
    \end{equation}
  \end{lemma}
  \begin{proof}
    This lemma is deduced with standard methods for parabolic equations, see e.g. \cite{LSU},
    and we leave most of its proof to the interested reader.
    Here, we shall only comment on the well-definiteness of the flow
    on \GGG $\regdens$. \EEE
    From the classical theory, it follows that $\aux\in C^\infty(\setR_+\times\Omega)$,
    i.e., the generalized solution is smooth and classical for positive times, and is continuous initially.
    One then easily verifies, using \eqref{MobL}, that $\aux$ satisfies the comparison principle.
    Next, observe that \eqref{eq.claw2} admits a family of stationary solutions $(\tilde\aux^c)_{c\in\setR}$
    of the form
    \begin{align*}
      \tilde\aux^c (x) = F^{-1}\big(c-\eps^{-1}V(x)\big)
      \quad\text{with}\quad
      F(s) = \int_{s_0}^s \frac{\d\sigma}{\mob(\sigma)}.
    \end{align*}
    It is easily seen from the properties of $\mob$ and by smoothness of $V$,
    that $\tilde\aux^c\in\regdens$ for every $c\in\setR$,
    and that $\tilde\aux^c\downarrow0$ or $\tilde\aux^c\uparrow M$ uniformly on $\Omega$,
    for $c\downarrow-\infty$ or $c\uparrow\infty$, respectively.
    Hence, any solution $\aux$ can be sandwiched between two stationary solutions.
    Since $\aux_0\in \regdens$,
    one has $\aux(s)\in\regdens$ for all $s>0$.

\end{proof}

\end{appendix}

\end{document}